\pdfoutput=1

\documentclass[paper=a4, pagesize, fontsize=12pt, twoside=true, DIV=calc, bibliography=totoc, final, version=last]{scrartcl}

\usepackage[T1]{fontenc}
\usepackage[utf8]{inputenc}
\usepackage[USenglish]{babel}
\usepackage{lmodern}

\usepackage[bibstyle=alphabetic, citestyle=alphabetic, backend=biber]{biblatex} %
\usepackage{csquotes} %

\usepackage{typearea}
\usepackage[automark]{scrlayer-scrpage}
\setkomafont{disposition}{\scshape}
\setkomafont{descriptionlabel}{\normalfont\itshape}
\usepackage{mathtools} %
	\mathtoolsset{mathic} %
\usepackage{amsfonts}
\usepackage{amsthm}
\usepackage{amssymb}
\usepackage{bbold} %
\usepackage{mathdots} %
\usepackage{braket} %
\usepackage{xfrac} %
\UseCollection{xfrac}{plainmath}
\usepackage{tikz-cd}
	\usetikzlibrary{arrows}
\KOMAoptions{captions=heading}

\usepackage{microtype}

\usepackage{hyperxmp} %
\usepackage[hidelinks, unicode=true, linktoc=all, pdfencoding=unicode]{hyperref} %
\hypersetup{
	pdfauthor={Enrico Ghiorzi},
	pdftitle={Complete internal categories},
	pdfsubject={Category Theory},
	pdfcopyright={Copyright \copyright\ 2020 Enrico Ghiorzi},
	pdfkeywords={category theory, internal categories, limits, completeness},
	pdftype={Text}
}
\usepackage{bookmark} %
\usepackage{cleveref} %

\recalctypearea

\theoremstyle{plain}
\newtheorem{theorem}{Theorem}[section]
\newtheorem{proposition}[theorem]{Proposition}

\theoremstyle{definition}
\newtheorem{definition}[theorem]{Definition}
\newtheorem{notation}[theorem]{Notation}

\theoremstyle{remark}
\newtheorem{example}[theorem]{Example}
\newtheorem{remark}[theorem]{Remark}

\crefname{axiom}{axiom}{axioms}
\crefname{example}{example}{examples}
\crefname{formula}{formula}{formulae}

\NewDocumentCommand\CatStyle{ m }{\mathcal{#1}}
\NewDocumentCommand\IntCatStyle{ m }{\mathbf{#1}}

\NewDocumentCommand{\Cat}		{ O{}O{} }	{\ensuremath{{#2}\mathrm{Cat}_{#1}}}
\NewDocumentCommand{\ICat}		{ O{}O{} }	{\ensuremath{{#2}\mathrm{ICat}_{#1}}}
\NewDocumentCommand{\MonCat}	{ O{} }		{\ensuremath{\mathrm{Mon}\Cat[#1]}}
\NewDocumentCommand{\SymMonCat}	{ O{} }		{\ensuremath{\mathrm{Sym}\MonCat[#1]}}
\NewDocumentCommand{\Multicat}	{ O{} }		{\ensuremath{\mathrm{Multicat}_{#1}}}
\NewDocumentCommand{\Graph}		{ O{} }		{\ensuremath{\mathrm{Graph}_{#1}}}
\NewDocumentCommand{\Cns}		{ m }		{\ensuremath{\mathrm{Cone}_{#1}}}
\NewDocumentCommand{\CoCns}		{ m }		{\ensuremath{\mathrm{CoCone}_{#1}}}
\NewDocumentCommand{\SetCat}{}				{\mathrm{Set}}

\NewDocumentCommand\Id			{ O{} }		{\mathrm{id}_{#1}}
\NewDocumentCommand\Source		{ O{} }		{\mathrm{s}_{#1}}
\NewDocumentCommand\Target		{ O{} }		{\mathrm{t}_{#1}}
\NewDocumentCommand\Hom			{ O{} } 	{#1}%
\NewDocumentCommand\Comp		{ O{} }		{\mathbin{\circ_{#1}}}
\NewDocumentCommand\VComp		{ O{} }		{\Comp[#1]}
\NewDocumentCommand\HComp		{ O{} }		{\mathbin{\ast_{#1}}}
\NewDocumentCommand\MonProd		{O{}O{}}	{\mathbin{\otimes^{#1}_{#2}}}
\NewDocumentCommand\MonUnit		{ O{} }		{\mathbb{I}_{#1}}
\NewDocumentCommand\MonAssoc	{ O{} }		{\alpha_{#1}}
\NewDocumentCommand\MonUnitorL	{ O{} }		{\lambda_{#1}}
\NewDocumentCommand\MonUnitorR	{ O{} }		{\rho_{#1}}
\NewDocumentCommand\MonSym		{ O{} }		{\sigma_{#1}}
\NewDocumentCommand\Prod		{ O{} } 	{\mathbin{\times^{#1}}}
\NewDocumentCommand\Terminal	{ O{} } 	{\mathbb{1}_{#1}}
\NewDocumentCommand\Initial		{ O{} } 	{\mathbb{0}_{#1}}
\NewDocumentCommand\Pullback 	{O{}O{}}	{\mathbin{\prescript{}{#1}{\Prod}_{#2}}}
\NewDocumentCommand\Diag		{ O{} } 	{\delta_{#1}}
\NewDocumentCommand\Aug			{ O{} } 	{\mathrm{e}_{#1}}
\NewDocumentCommand\Inv			{ O{} } 	{\mathrm{i}_{#1}}
\NewDocumentCommand\Adjoint		{}			{\dashv}
\NewDocumentCommand\Op			{}			{\mathbf{op}}

\DeclareMathOperator\Ind					{ind} %
\DeclareMathOperator\Dis					{dis} %

\NewDocumentCommand{\TiC}	{}{\vdash}

\DeclarePairedDelimiter{\ExternalizationBrk}{\lbrack}{\rbrack}
\NewDocumentCommand{\Externalization} {m O{}} {\ExternalizationBrk{#1}^{#2}}

\NewDocumentCommand{\intCat}	{O{}}	{\ensuremath{\Cat_{#1}}}

\NewDocumentCommand\Fun			{mm}	{\mathrm{Fun}(#1, #2)}

\NewDocumentCommand\Nat			{mm}	{\mathrm{Nat}(#1, #2)}

\NewDocumentCommand\EndDiag		{O{}}	{\mathrm{End}(#1)}

\NewDocumentCommand\DoubleCat	{}		{\mathbb{D}}
\NewDocumentCommand\Parallel	{m}		{#1_{\mathrm{p}}}
\NewDocumentCommand\ParallelCat	{}		{\mathbb{P}}
\NewDocumentCommand{\idMon}		{ O{} } {\mathrm{id}_{#1}} %

\DeclarePairedDelimiter{\Name}{\ulcorner}{\urcorner}
\DeclarePairedDelimiterXPP{\nStEq}[3]{}{\lbrack}{\rbrack}{}{#1 =_{#3} #2}
\DeclarePairedDelimiterXPP{\nStIn}[2]{}{\lbrack}{\rbrack}{}{#1 \in #2}

\DeclareMathOperator{\Lim}			{lim}

\NewDocumentCommand{\Eval}			{O{}}	{\mathrm{e}}%
\DeclarePairedDelimiterX{\PowObj}[2]{\lbrack}{\rbrack}{#1, #2}
\DeclarePairedDelimiterX{\WLim}[2]{\lbrace}{\rbrace}{#2; #1}

\DeclarePairedDelimiterX{\pairEnc}[2]{\langle}{\rangle}{#1, #2}
\DeclarePairedDelimiterX{\PSh}[2]{\lbrack}{\rbrack}{#1, #2}

\NewDocumentCommand \Def			{ m }	{\emph{#1}}

\NewDocumentCommand \Iso			{ }		{\mathrel{\cong}}

\NewDocumentCommand \Entry			{ O{} } {\mathhyphen_{#1}}

\NewDocumentCommand{\DefEq} {} {\coloneq}%
\DeclarePairedDelimiterX{\pair}[2]{\lparen}{\rparen}{#1, #2}

\NewDocumentCommand{\N} {} {\mathbb{N}}

\NewDocumentCommand{\mathhyphen}	{}{\text{-}}

\begin{document}
	
\title{Complete Internal Categories}
\author{Enrico Ghiorzi\thanks{Research funded by Cambridge Trust and EPSRC.}}
\date{April 2020}
	
\maketitle
	
\abstract{
Internal categories feature notions of limit and completeness, as originally proposed in the context of the effective topos.
This paper sets out the theory of internal completeness in a general context, spelling out the details of the definitions of limit and completeness and clarifying some subtleties.
Remarkably, complete internal categories are also cocomplete and feature a suitable version of the adjoint functor theorem.
Such results are understood as consequences of the intrinsic smallness of internal categories.
}
	
\tableofcontents
	
\NewDocumentCommand{\E} {} {\CatStyle{E}}

\addsec{Introduction} {

It is a remarkable feature of the effective topos \autocite{hyland1982effective} to contain a small subcategory, the category of modest sets, which is in some sense complete.
That such a category should exist was originally suggested by Eugenio Moggi as a way to understand how realizability toposes give rise to models for impredicative polymorphism,
and concrete versions of such models already appeared in \cite{girard1972interpretation}.
The sense in which the category of modest sets is complete is a delicate matter which,
among other aspects of the situation, is considered in \cite{Hyland90discreteobjects}.
This notion of (strong) completeness is related in a reasonably straightforward way---using the externalization of an internal category---to the established notion for indexed categories \autocite{pare1978abstract}.
There is a rough sketch of the indexed point of view in \cite{hyland1982effective}.

Sadly, the existing literature on internal complete categories is sparse.
The paper \cite{hyland1988small}, which presents a leading example, gives a sketch of how the theory might develop.
The more or less contemporaneous paper \cite{Hyland90discreteobjects} discusses the definitions in light of a perspective suggested by Freyd, but its main focus is on weak notions of completeness.
The basic idea is that one has a weak limit when it is internally true that there exists a limit cone for the given diagram, while the limit is strong when the choice of limit cone is given as part of the structure.
In this paper, though, we are not assuming that the internal logic of our ambient category features existential quantification, so the notion of weak limit is not even applicable.

In this paper, we adopt and clarify the definition of strong completeness given in \cite{Hyland90discreteobjects}, where the use of the internal logic is intended to ensure that the property of being a limit cone is stable under pullback.
Despite making the definition of completeness very concise, such a choice has the downside of making the content less accessible.
Instead, we will present the notion in a more explicit way, by avoiding the use of the internal language in its definition.
Moreover, we shall prove some results for internal complete categories that do not generally hold for standard categories, suggesting that the absence of size issues makes internal complete categories better behaved than external ones: in particular, that completeness and cocompleteness in the internal context are equivalent, and that a suitable version of the adjoint functor theorem holds without requiring any solution set condition.
}

\section{Background}
\label{sec:background}

In this section we quickly recall, without any claim of completeness, some topics in Category Theory which will be needed as background.
Although these topics are standard and well-known, it is useful to spell them out anyway to set the notation.
Aside from that, we assume the reader to be familiar with the basic notions of Category Theory.
In this respect, we shall regard \textcite{MacLane98CategoriesWM,Borceaux94HandbookCA} as our main references.

In the context of this section, let \(\E\) be a category with finite limits, which we regard as our ambient category.
We also require \(\E\) to have a cartesian monoidal structure, that is, a monoidal structure given by a functorial choice of binary products \(\Prod \colon \E \Prod \E \to \E\) and a chosen terminal object \(\Terminal\).

Notice that, as a category with finite limits, \(\E\) is a model for cartesian logic, or finite limit logic.
So, we will frequently use its internal language to ease the notation.
The internal language will be extended to typed lambda-calculus when we will further assume \(\E\) to be locally cartesian closed.
There are multiple accounts of the internal language of categories in the literature.
In particular, we shall follow \textcite{johnstone2002sketches,crole1993categories}, but, since we only make a basic use of the internal language, other references would be equally adequate.

\subsection{Internal Categories} \label{sec:int_cat} {

\NewDocumentCommand{\A} {} {\IntCatStyle{A}}
\ProvideDocumentCommand{\B} {} {\IntCatStyle{B}}
\RenewDocumentCommand{\C} {} {\IntCatStyle{C}}
\NewDocumentCommand{\D} {} {\IntCatStyle{D}}
\NewDocumentCommand{\F} {} {\symcal{F}}
\NewDocumentCommand{\Fs} {} {\F_{\symrm{s}}}
\NewDocumentCommand{\I} {} {\CatStyle{I}}
\ProvideDocumentCommand{\M} {} {\CatStyle{M}}
\NewDocumentCommand{\V} {} {\CatStyle{V}}
\NewDocumentCommand{\X} {} {\CatStyle{X}}
\NewDocumentCommand{\Y} {} {\CatStyle{Y}}
\NewDocumentCommand{\W} {} {\CatStyle{W}}
\NewDocumentCommand{\Z} {} {\CatStyle{Z}}

We start by giving the definitions of internal category, functor and natural transformation using the internal language of \(\E\) as described before.

\begin{definition}[internal category]
	An \Def{internal category} \(\A\) in \(\E\) is a diagram
	\begin{equation*}
		\begin{tikzcd}[column sep = large]
			A_0
				\ar[r, "{\Id[\A]}" description]
			& A_1
				\ar[l, "{\Source[\A]}"', shift right=2, bend right]
				\ar[l, "{\Target[\A]}", shift left=2, bend left]
			& A_1 \Pullback[\Source][\Target] A_1
				\ar[l, "{\Comp[\A]}"']
		\end{tikzcd}
	\end{equation*}
	in \(\E\) (where \(A_1 \Pullback[\Source][\Target] A_1\) is the pullback of \(\Source[\A]\) and \(\Target[\A]\)) satisfying the usual axioms for categories, which can be expressed in the internal language of \(\E\) as follows (where, as usual, we use an infix notation for \(\Comp[\A]\)).
	\begin{align*}
		a \colon A_0 &\TiC \Source[\A] \Id[\A](a) = a \colon A_0 \\ %
		a \colon A_0 &\TiC \Target[\A] \Id[\A](a) = a \colon A_0 \\ %
		(g, f) \colon A_1 \Pullback[\Source][\Target] A_1 &\TiC \Source[\A] (g \Comp[\A] f) = \Source[\A] (f) \colon A_0 \\ %
		(g, f) \colon A_1 \Pullback[\Source][\Target] A_1 &\TiC \Target[\A] (g \Comp[\A] f) = \Target[\A] (g) \colon A_0 \\ %
		(h, g, f) \colon A_1 \Pullback[\Source][\Target] A_1 \Pullback[\Source][\Target] A_1
			&\TiC h \Comp[\A] (g \Comp[\A] f) = (h \Comp[\A] g) \Comp[\A] f \colon A_1 \\ %
		f \colon A_1 &\TiC f \Comp[\A] \Id[\A] \Source[\A](f) = f \colon A_1 \\ %
		f \colon A_1 &\TiC \Id[\A] \Target[\A](f) \Comp[\A] f = f \colon A_1 %
	\end{align*}
\end{definition}

While correct, the use we made of the internal language is quite unwieldy, especially when dealing with terms whose type is the object of arrows of an internal category.
We then introduce conventions to ease the use of the internal language, by bringing it closer to the standard notation of category theory.

\begin{notation}
	Given terms \(x \colon X \TiC t_0(x) \colon A_0\), \(x \colon X \TiC t_1(x) \colon A_0\) and \(x \colon X \TiC f(x) \colon A_1\), we shall write
	\[
		x \colon X \TiC f(x) \colon t_0(x) \to t_1(x)
		\quad \text{ or } \qquad
		x \colon X \TiC t_0(x) \xrightarrow{f(x)} t_1(x)
	\]
	instead of (the conjunction of) the formulas \(x \colon X \TiC \Source[\A]f(x) = t_0(x) \colon A_0\) and \(x \colon X \TiC \Target[\A]f(x) = t_1(x) \colon A_0\).
	
	Moreover, given terms \(x \colon X \TiC t_2(x) \colon A_0\) and \(x \colon X \TiC g(x) \colon A_1\) such that \(x \colon X \TiC g(x) \colon t_1(x) \to t_2(x)\), we shall write
	\[
		x \colon X \TiC t_0(x) \xrightarrow{f(x)} t_1(x) \xrightarrow{g(x)} t_2(x)
	\]
	instead of the term \(x \colon X \TiC g(x) \Comp[\A] f(x) \colon A_1\).
	Then, we can use the familiar notation for commuting diagrams even in the internal language.
	For example, in context \(x \colon X\), the commutativity of the diagram
	\[
		\begin{tikzcd}
			t_0(x)
					\ar[r, "f(x)"]
					\ar[d, "h(x)"']
				&t_1(x)
					\ar[d, "k(x)"] \\
			t_2(x)
					\ar[r, "g(x)"']
				&t_3(x)
		\end{tikzcd}
	\]
	can be translated into a meaningful formula of the internal language, predicating the equality of the terms given by composition.
	
	Finally, as it is common in standard categorical practice, we use the notation to set the source and target of arrows: when we have a term \(a_0, a_1 \colon A_0, f \colon A_1 \TiC t(a_0, a_1, f) \colon B\), we shall write
	\[
		f \colon t_0 \to_{\A} t_1 \TiC t(t_0, t_1, f) \colon B
	\]
	in place of the term \(f \colon A_1 \TiC t(\Source[\A](f), \Target[\A](f), f) \colon B\).
	Matching sources and targets are to be intended as being given by a pullback, so that we can write \(f \colon a_0 \to_{\A} a_1, g \colon a_1 \to_{\A} a_2\) for the context \((f, g) \colon A_1 \Pullback[\Source][\Target] A_1\).
\end{notation}

Leveraging the notation just introduced, we define functors of internal categories with relative compositions and identities, so that internal categories and their functors form a category.

\begin{definition}[internal functor]
	Let \(\A\) and \(\B\) be internal categories in \(\E\).
	A \Def{functor} of internal categories \(F \colon \A \to \B\) is given by a pair of arrows \(F_0 \colon A_0 \to B_0\) and \(F_1 \colon A_1 \to B_1\) such that
	\begin{align*}
		f \colon a_0 \to_{\A} a_1 &\TiC F_1(f) \colon F_0(a_0) \to_{\B} F_0(a_1), \\
		a \colon A_0 &\TiC F_1 \Id[\A](a) = \Id[\B] F_0(a) \colon B_1,
	\end{align*}
	and, in context \(f \colon a_0 \to_{\A} a_1, g \colon a_1 \to_{\A} a_2\), the diagram
	\[
		\begin{tikzcd}
			&F_0(a_1)
				\ar[dr, "{F_1(g)}"] \\
			F_0(a_0)
				\ar[ur, "{F_1(f)}"]
				\ar[rr, "{F_1(g \Comp[\A] f)}"']
			&&F_0(a_2)
		\end{tikzcd}
	\]
	commutes.
\end{definition}

The composition \(G F\) of two consecutive internal functors \(\A \xrightarrow{F} \B \xrightarrow{G} \C\) is given by the compositions \(G_0 F_0\) and \(G_1 F_1\); in the internal language, the definition reads
\begin{align*}
	a \colon A_0 &\TiC {(G F)}_0 (a) \DefEq G_0 F_0 (a) \colon C_0 \\
	f \colon A_1 &\TiC {(G F)}_1 (f) \DefEq G_1 F_1 (f) \colon C_1
\end{align*}
and it is associative.

The identity functor \(\Id[\A] \colon \A \to \A\) for an internal category \(\A\) is given by the identity arrows on \(A_0\) and \(A_1\); in the internal language, the definition reads
\begin{align*}
	a \colon A_0 &\TiC {(\Id[\A])}_0 (a) \DefEq a \colon A_0 \\
	f \colon A_1 &\TiC {(\Id[\A])}_1 (f) \DefEq f \colon A_1.
\end{align*}
With respect to such identity functors, the composition of internal functors satisfies the unit laws.

The following result sums up the content of the previous definitions.

\begin{proposition}
	Internal categories in \(\E\) and their functors, with their composition and identities, form a category \(\Cat[\E]\).
\end{proposition}

The previous proposition is a generalization of the standard (in the sense of external) analogous result for small categories, which indeed is the special case in which \(\E\) is \(\SetCat\).
Even the proof parallels that of the set-theoretic result, but carried out in the internal language of \(\E\).

The category of internal categories is well-behaved with respect to slicing, as the following remark makes clear.

\begin{remark}
	Let \(\E'\) be another category with finite limits, and \(F \colon \E \to \E'\) a functor preserving finite limits.
	Then, there is a functor \(F \colon \Cat[\E] \to \Cat[\E']\) (with abuse of notation) applying \(F\) to the underlying graph of internal categories.
\end{remark}

In the following remark, we notice some useful properties of \(\Cat[\E]\) in relation to slicing and change of base.

\begin{remark}
	\label{rmk:cat-reindexing}
	Let \(i \colon J \to I\) be an arrow in \(\E\).
	Then, there is an adjunction \(i_! \Adjoint i^* \colon \sfrac{\E}{I} \to \sfrac{\E}{J}\) where the functor \(i_! \colon \sfrac{\E}{J} \to \sfrac{\E}{I}\) is given by post-composition with \(i\), and the functor \(i^* \colon \sfrac{\E}{I} \to \sfrac{\E}{J}\) is given by pullback along \(i\).
	This adjunction extends to internal categories, yielding \(i_! \Adjoint i^* \colon \Cat[\sfrac{\E}{I}] \to \Cat[\sfrac{\E}{J}]\).
	In particular, the unique arrow \(!_I \colon I \to \Terminal\) yields an adjunction \(I_! \Adjoint I^* \colon \Cat[\E] \to \Cat[\sfrac{\E}{I}]\).
\end{remark}

We now define natural transformations of internal categories, with relative identities and horizontal and vertical compositions, so that internal categories, their functors and the natural transformations between them shall form a 2-category.

\begin{definition}[internal natural transformation]
	Let \(F, G \colon \A \to \B\) be functors of internal categories in \(\E\).	
	A \Def{natural transformation} of internal functors \(\alpha \colon F \to G \colon \A \to \B\) is given by an arrow \(\alpha \colon A_0 \to B_1\) such that
	\[
		a \colon A_0 \TiC \alpha_{a} \colon F_0(a) \to_{\B} G_0(a)
	\]
	and, in context \(f \colon a \to_{\A} a'\), the diagram
	\[
		\begin{tikzcd}[ampersand replacement = \&]
			F_0(a)
					\ar[r, "{\alpha_{a}}"]
					\ar[d, "{F_1(f)}"']
				\&G_0(a)
					\ar[d, "{G_1(f)}"] \\
			F_0(a')
					\ar[r, "{\alpha_{a'}}"']
				\&G_0(a')
		\end{tikzcd}
	\]
	commutes.
	Notice that the argument of the natural transformation is sub-fixed, following the standard conventions of category theory.
\end{definition}
	
We define the vertical and horizontal compositions of natural transformations in
\[
	\begin{tikzcd}[column sep = large]
		\C
			\ar[r, "L"]
		& \A
			\ar[rr, shift left, bend left = 45, "F", ""{name=F, below}]
			\ar[rr, "G" description, ""{name=Ga, above}, ""{name=Gb, below}]
			\ar[rr, shift right, bend right = 45, "H"', ""{name=H, above}]
		&& \B
			\arrow[Rightarrow, from=F, to=Ga, "\alpha"]
			\arrow[Rightarrow, from=Gb, to=H, "\beta"]
			\ar[r, "{R}"]
		& \D
	\end{tikzcd}
\]
as
\begin{align*}
	a \colon A_0 &\TiC {(\alpha \beta)}_{a} \DefEq F_0(a) \xrightarrow{\alpha_{a}} G_0(a) \xrightarrow{\beta_{a}} H_0(a) \\
	c \colon C_0 &\TiC {(\alpha L)}_{a} \DefEq F_0 L_0 (c) \xrightarrow{\alpha_{L(c)}} G_0 L_0 (c) \\
	a \colon A_0 &\TiC {(R \alpha)}_{a} \DefEq R_0 F_0 (a) \xrightarrow{R_1(\alpha_{a})} R_0 G_0 (a).
\end{align*}
Such compositions satisfy the interchange law and the vertical one is associative.

The identity natural transformation \(\Id[F] \colon F \to F \colon \A \to \B\) is defined by
\[
	a \colon A_0 \TiC ({\Id[F])}_{a} \DefEq F_0(a) \xrightarrow{\Id[\A] F_0 (a)} F_0(a).
\]
With respect to such identity, the compositions of internal natural transformations satisfy the unit laws.

The following result sums up the content of the previous definitions.

\begin{proposition}
	Internal categories in \(\E\), their functors, and natural transformations between them, with their composition and identities, form a 2-category \(\Cat[\E]\) (denoted in the same way as its underlying 1-category with abuse of notation).
\end{proposition}

Again, the previous proposition is a generalization of the standard (in the sense of external) analogous result for small categories.

The following result is the internal version of the standard set-theoretic one, and it can be proved by a completely routine application of the internal language of \(\E\).

\begin{proposition}
	The category \(\Cat[\E]\) has finite limits induced point-wise by the corresponding limits in \(\E\).
	In particular, there is a terminal internal category \(\Terminal[\Cat[\E]]\) and a binary product \(\Prod[\Cat[\E]]\) of internal categories making \(\Cat[\E]\) a cartesian monoidal category.
\end{proposition}

There is also an obvious underlying-object-of-objects functor \(\Cat[\E] \to \E\), given in the next definition, preserving the cartesian monoidal structure.

\begin{definition}
	The \Def{objects functor} is the monoidal functor \(U \colon \Cat[\E] \to \E\) sending an internal category \(\A\) into its object of objects \(A\), and an internal functor \(F \colon \A \to \B\) into its object component \(F_0 \colon A \to B\).
\end{definition}

We now present a few remarkable examples of internal categories.
	
\begin{example}
	\label{example:discrete-category}
	Let \(A\) be an object of \(\E\).
	The \Def{discrete category} \(\Dis A\) over \(A\) is given by
	\begin{equation*}
		\begin{tikzcd}[column sep = huge]
			\begin{multlined} {(\Dis A)}_0 \\ = A \end{multlined}
				\ar[r, "{\Id[{\Dis A}] = \Id[A]}"]
			& \begin{multlined} {(\Dis A)}_1 \\ = A \end{multlined}
				\ar[l, bend right = 15, "{\Source[{\Dis A}] = \Id[A]}"', shift right=2]
				\ar[l, bend left = 15, "{\Target[{\Dis A}] = \Id[A]}", shift left=2]
			& \begin{multlined} {(\Dis A)}_1 \Pullback[\Source][\Target] {(\Dis A)}_1 \\ \Iso A \end{multlined}
				\ar[l, "{\Comp[{\Dis A}] = \Id[A]}"']
		\end{tikzcd}
	\end{equation*}
	An alternative, more abstract way to look at the discrete category with respect to \Cref{rmk:cat-reindexing} is to notice that \(\Dis A\) is (equivalent to) \(A_! A^* \Terminal[\Cat[\E]]\).
	This construction extends to a monoidal functor \(\Dis \colon \E \to \Cat[\E]\).
\end{example}

\begin{example}
	Let \(A\) be an object of \(\E\).
	The \Def{indiscrete category} \(\Ind A\) over \(A\) is given by
	\begin{equation*}
		\begin{tikzcd}[column sep = huge]
			\begin{multlined} {(\Ind A)}_0 \\ = A \end{multlined}
				\ar[r, "{\Id[{\Ind A}] = \Delta_A}"]
			& \begin{multlined} {(\Ind A)}_1 \\ = A \Prod A \end{multlined}
				\ar[l, bend right = 15, "{\Source[{\Ind A}] = \pi_1}"', shift right=2]
				\ar[l, bend left = 15, "{\Target[{\Ind A}] = \pi_2}", shift left=2]
			& \begin{multlined}
				{(\Ind A)}_1 \Pullback[\Source][\Target] {(\Ind A)}_1 \\
				\Iso A \Prod A \Prod A
			\end{multlined}
				\ar[l, "{\!\!\!\Comp[{\Ind(A)}]\!=\!(\pi_1,\!\pi_3\!)}"']
		\end{tikzcd}
	\end{equation*}
	This construction extends to a monoidal functor \(\Ind \colon \E \to \Cat[\E]\).
\end{example}

The above constructions yield the free and co-free internal categories over an object of \(\E\), as formally stated by the following proposition, whose proof is, again, routine.

\begin{proposition}
	\label{prop:dis-U-ind}
	There are monoidal adjunctions \(\Dis \Adjoint U \Adjoint \Ind\).
\end{proposition}

As a further example, consider opposite categories.

\begin{example}
	Let \(\A\) be an internal category in \(\E\).
	The opposite category \(\A^{\Op}\) is obtained by switching the source and target arrows of \(\A\), so that \(\Source[\A^{\Op}] = \Target[\A]\) and \(\Target[\A^{\Op}] = \Source[\A]\).
	This construction extends to a monoidal functor \({(\mathhyphen)}^{\Op} \colon \Cat[\E] \to \Cat[\E]\).
\end{example}

Then, it is not difficult to prove the following result.

\begin{proposition}
	\label{prop:op-self-adj}
	The monoidal functor \({(\mathhyphen)}^{\Op}\) is a self-adjoint automorphism.
\end{proposition}

}

\subsection{Exponentials of Internal Categories}
\label{sec:exp-int-cats} {

\NewDocumentCommand{\A} {} {\IntCatStyle{A}}
\ProvideDocumentCommand{\B} {} {\IntCatStyle{B}}
\RenewDocumentCommand{\C} {} {\IntCatStyle{C}}
\NewDocumentCommand{\D} {} {\IntCatStyle{D}}
\NewDocumentCommand{\F} {} {\symcal{F}}
\NewDocumentCommand{\Fs} {} {\F_{\symrm{s}}}
\NewDocumentCommand{\I} {} {\CatStyle{I}}
\ProvideDocumentCommand{\M} {} {\CatStyle{M}}
\NewDocumentCommand{\V} {} {\CatStyle{V}}
\NewDocumentCommand{\X} {} {\CatStyle{X}}
\NewDocumentCommand{\Y} {} {\CatStyle{Y}}
\NewDocumentCommand{\W} {} {\CatStyle{W}}
\NewDocumentCommand{\Z} {} {\CatStyle{Z}}

We shall present the cartesian closed structure of \(\Cat[\E]\), which is a generalization of the cartesian closed structure of \(\Cat\).

Remember that the standard construction of categories of functors makes essential use of sets of functions.
In light of that, we need to assume that the ambient category \(\E\) is locally cartesian closed, meaning that its internal language will be the simply typed lambda calculus extension of finite limit logic.
In that regard, it shall be useful to introduce the following notation.

\begin{notation}
	If \(A\) and \(B\) are objects of \(\E\), let \(\Eval[\E] \colon B^A \Prod A \to B\) be the evaluation morphism.
	The notation for the evaluation arrow should also be decorated with \(A\) and \(B\), but it is unnecessary as those are usually clear from the context.
\end{notation}

It is also useful to introduce a convention to denote subobjects yielded by an equalizer.

\begin{notation}
	Given two terms, \(a \colon A \TiC t(a) \colon B\) and \(a \colon A \TiC t'(a) \colon B\), the subobject of \(A\) of those \(a \colon A\) such that \(t(a) = t'(a)\) is given by the equalizer of the two parallel arrows \(A \to B\) yielded by \(t\) and \(t'\).
\end{notation}

The functors between two categories form a set, suggesting that \(\Cat\) is enriched over \(\SetCat\).
Analogously, we expect \(\Cat[\E]\) to be enriched over \(\E\), in the sense made precise by the statement of the following proposition.
Then, the theory of categories internal to \(\E\) becomes clearer from the perspective of enriched category theory.

\begin{proposition}
	\label{prop:CatE-enriched}
	There is a category enriched in \(\E\) (which shall be called \(\Cat[\E]\) with abuse of notation) whose underlying category is isomorphic to \(\Cat[\E]\) itself.
\end{proposition}
\begin{proof}
	The hom-object \(\Hom[\Cat[\E]](\A, \B)\) of internal categories \(\A\) and \(\B\) represents the functors \(\A \to \B\), and is defined as the subobject of \(B_0^{A_0} \Prod B_1^{A_1}\) of those \(F = (F_0, F_1)\) satisfying the functoriality axioms, i.e.\ such that
	\[
		\lambda f \colon a_0 \to_{\A} a_1.\ \Eval[\E](F_0, a_0) \xrightarrow{\Eval[\E](F_1, f)} \Eval[\E](F_0, a_1)
	\]
	and
	\[
		\begin{multlined}
			\lambda f \colon a_0 \to_{\A} a_1, g \colon a_1 \to_{\A} a_2.\ \Eval[\E](F_1, g \Comp[\A] f) \\
			=\lambda f \colon a_0 \to_{\A} a_1, g \colon a_1 \to_{\A} a_2.\ \Eval[\E](F_1, g) \Comp[\B] \Eval[\E](F_1, f).
		\end{multlined}
	\]
	
	The composition of internal categories \(\A\), \(\B\) and \(\C\) is the arrow
	\[
		\Comp[\Cat[\E]](\A, \B, \C) \colon \Hom[\Cat[\E]](\B, \C) \Prod \Hom[\Cat[\E]](\A, \B) \to  \Hom[\Cat[\E]](\A, \C)
	\]
	defined in context \(F \colon \Hom[\Cat[\E]](\A, \B)\), \(G \colon \Hom[\Cat[\E]](\B, \C)\) as
	\begin{gather*}
		\begin{multlined}
				{\Comp[\Cat[\E]](\A, \B, \C)(F, G)}_0
				\DefEq \lambda a \colon A_0.\ \Eval[\E] \big( G_0, \Eval[\E](F_0, a) \big)
				\colon {C_0}^{A_0}
		\end{multlined} \\
		\begin{multlined}
				{\Comp[\Cat[\E]](\A, \B, \C)(F, G)}_1
				\DefEq \lambda f \colon A_1.\ \Eval[\E] \big( G_1, \Eval[\E](F_1, f) \big)
				\colon {C_1}^{A_1}
		\end{multlined}
	\end{gather*}
	and the identity of an internal category is defined analogously.

	The verification that these data give an enrichment is a simple exercise in the internal language, and the points of the hom-object \(\Hom[\Cat[\E]](\A, \B)\) are evidently in bijective correspondence with functors \(\A \to \B\).
\end{proof}

The theory developed in \Cref{sec:int_cat} extends to the context of enriched category theory.
For example, consider the following proposition.

\begin{proposition}
	The adjunctions from \Cref{prop:dis-U-ind} are \(\E\)-enriched. %
\end{proposition}

Likewise, we get an enriched version of \Cref{rmk:cat-reindexing}.

\begin{remark}
	Let \(i \colon J \to I\) be an arrow in \(\E\).
	Then, \(\Cat[\sfrac{\E}{I}]\) is enriched on \(\sfrac{\E}{J}\) through the change of base \(i^* \colon \sfrac{\E}{I} \to \sfrac{\E}{J}\), and \(i_! \Adjoint i^* \colon \Cat[\sfrac{\E}{I}] \to \Cat[\sfrac{\E}{J}]\) is an adjunction of \(\sfrac{\E}{J}\)-enriched functors.
\end{remark}

We can finally prove that \(\Cat[\E]\) is cartesian closed.
Moreover, as it is clear from the following proof, if \(\A\) and \(\B\) are categories in \(\E\), then the points of \(\B^{\A}\) are in bijective correspondence with the functors \(\A \to \B\).

\begin{proposition}
	\label{prop:CatE-cart-closed}
	The category \(\Cat[\E]\) is cartesian closed.
\end{proposition}
\begin{proof}
	Let \(\A\) and \(\B\) be internal categories of \(\E\).
	We define the exponential object \(\B^{\A}\), representing the internal category of functors \(\A \to \B\) and natural transformations between them:
	\begin{description}
		\item[Object of objects] \({(\B^{\A})}_0\) is \(\Hom[\Cat[\E]](\A, \B)\), as defined in \Cref{prop:CatE-enriched}.
		\item[Object of arrows] \({(\B^{\A})}_1\) is the subobject of \({(\B^{\A})}_0 \Prod {(\B^{\A})}_0 \Prod B_1^{A_0}\) given, in context \(\big( F \colon {(\B^{\A})}_0, G \colon {(\B^{\A})}_0, \alpha \colon B_1^{A_0} \big)\), by the axiom
			\[
				\lambda f \colon a_0 \to_{\A} a_1.\
				\begin{tikzcd}[ampersand replacement = \&]
					\Eval[\E](F_0, a_0)
							\ar[r, "{\Eval[\E](\alpha, a_0)}"]
							\ar[d, "{\Eval[\E](F_1, f)}"']
						\&\Eval[\E](G_0, a_0)
							\ar[d, "{\Eval[\E](G_1, f)}"] \\
					\Eval[\E](F_0, a_1)
							\ar[r, "{\Eval[\E](\alpha, a_1)}"']
						\&\Eval[\E](G_0, a_1)
				\end{tikzcd}
			\]
			(that is, an equalizer with values in \(B_1^{A_1}\)).
		\item[Composition] is the arrow \(\Comp[\B^{\A}] \colon {(\B^{\A})}_1 \Pullback[\Source][\Target] {(\B^{\A})}_1 \to {(\B^{\A})}_1\) defined, in context \( \big( (G, H, \beta), (F, G, \alpha) \big) \colon {(\B^{\A})}_1 \Pullback[\Source][\Target] {(\B^{\A})}_1\), as
			\[
				(G, H, \beta) \Comp[\B^{\A}] (F, G, \alpha)
				\DefEq \big( F, H, \lambda a \colon A_0.\ \Eval[\E](\beta, a) \Comp[\B] \Eval[\E](\alpha, a) \big)
			\]
		\item[Identity] is the arrow \(\Id[\B^{\A}] \colon {(\B^{\A})}_0 \to {(\B^{\A})}_1\) defined, in context \(F \colon {(\B^{\A})}_0\), as
			\[
				\Id[\B^{\A}](F) \DefEq \big( F, F, \lambda a \colon A_0.\ \Id[\B](F_0(a)) \big).
			\]
	\end{description}
	
	Then, a standard argument shows that there is an isomorphism
	\[
		\Hom[\Cat[\E]](\A' \Prod[\Cat[\E]] \A, \B) \Iso \Hom[\Cat[\E]](\A', \B^{\A})
	\]
	natural in \(\A'\).
\end{proof}

Of course, the fact that \(\Cat[\E]\) is cartesian closed tells us that it is enriched in itself.
That extends the fact that it is enriched in \(\E\).
In fact, the enrichment in \(\E\) is essentially obtained from that in \(\Cat[\E]\) by change of base along the objects functor \(U \colon \Cat[\E] \to \E\).
}

\section{Completeness of Internal Categories}
\label{sec:completeness} {

\NewDocumentCommand{\A} {} {\IntCatStyle{A}}
\ProvideDocumentCommand{\B} {} {\IntCatStyle{B}}
\RenewDocumentCommand{\C} {} {\IntCatStyle{C}}
\NewDocumentCommand{\D} {} {\IntCatStyle{D}}
\NewDocumentCommand{\F} {} {\symcal{F}}
\NewDocumentCommand{\Fs} {} {\F_{\symrm{s}}}
\NewDocumentCommand{\I} {} {\IntCatStyle{I}}
\ProvideDocumentCommand{\M} {} {\CatStyle{M}}
\NewDocumentCommand{\V} {} {\CatStyle{V}}
\NewDocumentCommand{\X} {} {\CatStyle{X}}
\NewDocumentCommand{\Y} {} {\CatStyle{Y}}
\NewDocumentCommand{\W} {} {\CatStyle{W}}
\NewDocumentCommand{\Z} {} {\CatStyle{Z}}

This section introduces the main topics of the paper, namely internal limits and completeness of internal categories.
The definitions of such notions have to be considered carefully, for (in the standard external setting) they involve universal and existential quantifications, which may not be available in the internal language of the ambient category.
For this reason, it is not possible to simply translate the standard definitions in the internal language, and indeed, as we shall see, it is not possible to give purely internal definitions.
Still, it is paramount that the definitions are meaningful from the point of view of the internal logic.

Despite not being immediately obvious, it is necessary to assume that the ambient category \(\E\) is locally cartesian closed.
In particular, that means that the results from \Cref{sec:exp-int-cats} hold and that there are internal categories of functors.

\subsection{Diagrams and Cones}

As obvious as they might seem, we shall spell out the definitions of diagrams and cones over diagrams in the internal context.

\begin{definition}[diagram]
	Let \(\A\) and \(\D\) be internal categories in \(\E\).		
	A \Def{diagram} of shape \(\D\) in \(\A\) is an internal functor \(D \colon \D \to \A\).
\end{definition}

\begin{definition}[cone]
	A \Def{cone} over a diagram \(D \colon \D \to \A\) with \Def{vertex} \(V \colon \Terminal[\Cat[\E]] \to \A\)  is a natural transformation \(\gamma \colon \Delta V \to \Name{D} \colon \Terminal[\Cat[\E]] \to \PowObj{\D}{\A}\), where \(\Delta \colon \A \to \PowObj{\D}{\A}\) is the diagonal functor.
	The situation is shown in the following diagram.
	\[
		\begin{tikzcd}
			\Terminal[\Cat[\E]]
					\ar[r, equal]
					\ar[d, "{V}"']
				&\Terminal[\Cat[\E]]
					\ar[d, "{\Name{D}}"] \\
			\A
					\ar[r, "{\Delta}"']
					\ar[ur, Rightarrow, "{\gamma}"]
				&\PowObj{\D}{\A}
		\end{tikzcd}
	\]
	
	Let \(\gamma_i \colon V_i \Delta \to \Name{D} \colon \Terminal[\Cat[\E]] \to \PowObj{\D}{\A}\) for \(i = 0, 1\) be two cones over the diagram \(D \colon \D \to \A\) with tips \(V_i \colon \Terminal[\Cat[\E]] \to \A\) respectively.
	A morphism of cones \(h \colon \gamma_0 \to \gamma_1\) is given by a natural transformation \(h \colon V_0 \to V_1\) such that \(\gamma_1 \Comp (\Delta h) = \gamma_0\), as represented in the following diagram.
	\[
		\begin{tikzcd}[row sep = large, column sep = large]
			\Terminal[\Cat[\E]]
					\ar[r, equal]
					\ar[d, bend right, "{V_0}"' near start, ""{name=s, left}]
					\ar[d, bend left, "{V_1}" near start, ""'{name=t, right}]
					\ar[from=s, to=t, Rightarrow, "{h}"]
				&\Terminal[\Cat[\E]]
					\ar[d, "{\Name{D}}"] \\
			\A
					\ar[r, "{\Delta}"']
					\ar[ur, Rightarrow, "{\gamma_1}"']
				&\PowObj{\D}{\A}
		\end{tikzcd}
		=
		\begin{tikzcd}[row sep = large, column sep = large]
			\Terminal[\Cat[\E]]
					\ar[r, equal]
					\ar[d, "{V_0}"']
				&\Terminal[\Cat[\E]]
					\ar[d, "{\Name{D}}"] \\
			\A
					\ar[r, "{\Delta}"']
					\ar[ur, Rightarrow, "{\gamma_0}"]
				&\PowObj{\D}{\A}
		\end{tikzcd}
	\]
\end{definition}

Dually, we can give the definition of cocone.

\begin{definition}[cocone]
	A \Def{cocone} over a diagram \(D \colon \D \to \A\) with \Def{vertex} \(V \colon \Terminal[\Cat[\E]] \to~\A\) is a natural transformation \(\gamma \colon \Name{D} \to  \Delta V \colon \Terminal[\Cat[\E]] \to \PowObj{\D}{\A}\).
\end{definition}

All the previous considerations made for cones apply, dually, for cocones as well.

Let \(I\) be an object of \(\E\), to be intended as an indexing object.
We know from \Cref{rmk:cat-reindexing} that there are two adjunctions, \(I_! \Adjoint I^* \colon \E \to \sfrac{\E}{I}\) and \(I_! \Adjoint I^* \colon \Cat[\E] \to~\Cat[\sfrac{\E}{I}]\).
Moreover, the canonical morphism \(i \colon I^* \PowObj{\D}{\A} \to \PowObj{I^* \D}{I^* \A}\) is an isomorphism because \(\E\) is locally cartesian closed.
Then, given a diagram \(D \colon \D \to \A\), we get a diagram \(I^* D \colon I^* \D \to I^* \A\), and, given a cone
\[
	\begin{tikzcd}
		\Terminal[\Cat[\E]]
				\ar[r, equal]
				\ar[d, "{V}"']
			&\Terminal[\Cat[\E]]
				\ar[d, "{\Name{D}}"] \\
		\A
				\ar[r, "{\Delta}"']
				\ar[ur, Rightarrow, "{\gamma}"]
			&\PowObj{\D}{\A}
	\end{tikzcd}
\]
over \(D\), we get a cone \(i (I^* \gamma)\)
\[
	\begin{tikzcd}
		\Terminal[\Cat[\sfrac{\E}{I}]]
				\ar[r, equal]
				\ar[d, "{I^* V}"']
			&\Terminal[\Cat[\sfrac{\E}{I}]]
				\ar[d, "{I^* \Name{D}}"]
				\ar[r, equal]
			&\Terminal[\Cat[\sfrac{\E}{I}]]
				\ar[d, "{\Name{I^* D}}"] \\
		I^* \A
				\ar[r, "{I^* \Delta}"']
				\ar[rr, bend right, "{\Delta}"]
				\ar[ur, Rightarrow, "{I^* \gamma}"]
			&I^* \PowObj{\D}{\A}
				\ar[r, "{i}"']
			&\PowObj{I^* \D}{I^* \A}
	\end{tikzcd}
\]
over \(I^* D\).
Notice that \(I^*\) preserves the terminal object, \(i (I^* \Name{D}) = \Name{I^* D}\) and \(i (I^* \Delta) = \Delta\).
Then, we have a functor \(I^* \colon \Cns{D} \to \Cns{I^* D}\).

It is worth noting that the previous notions are all internal, in the sense that they can be expressed by the internal language of \(\E\).
Indeed, the internal category of cones over a diagram \(D \colon \D \to \A\) is given by the lax pullback
\[
	\begin{tikzcd}
		\Cns{D}
				\ar[r] \ar[d] \ar[dr, phantom, "\lrcorner"{description, very near start}]
			&\Terminal[\Cat[\E]]
				\ar[d, "{\Name{D}}"] \\
		\A
				\ar[r, "{\Delta}"']
				\ar[ur, Rightarrow]
			&\PowObj{\D}{\A}
	\end{tikzcd}
\]
(it is also easy to build \(\Cns{D}\) explicitly, thus showing the existence of such lax pullback).
Then, the category of points of \(\Cns{D}\) is the external category of cones over \(D\) and their transformations;
in particular, any cone \((V, \gamma)\) over \(D\) corresponds uniquely to a global section \((V, \gamma) \colon \Terminal[\Cat[\E]] \to \Cns{D}\).
Moreover, if \(I\) is an object of \(\E\), an \(I\)-indexed family of cones over \(D\) is given by a 2-cell
\begin{equation*}
	\begin{tikzcd}
		\I
				\ar[r] \ar[d, "V"']
			&\Terminal[\Cat[\E]]
				\ar[d, "{\Name{D}}"] \\
		\A
				\ar[r, "{\Delta}"']
				\ar[ur, Rightarrow, "\gamma"]
			&\PowObj{\D}{\A}
	\end{tikzcd}
\end{equation*}
where \(\I\) is \(I_! I^* \Terminal[\Cat[\E]]\) (intuitively, the discrete category over the object \(I\)).
Analogously, the category of all cones (over any diagram) of shape \(\D\) over \(\A\) is given by the following lax pullback.
\[
	\begin{tikzcd}
		\Cns{\D}
				\ar[r] \ar[d] \ar[dr, phantom, "\lrcorner"{description, very near start}]
			&\PowObj{\D}{\A}
				\ar[d, equal] \\
		\A
				\ar[r, "{\Delta}"']
				\ar[ur, Rightarrow]
			&\PowObj{\D}{\A}
	\end{tikzcd}
\]
Dually, the internal category of cocones over a diagram \(D \colon \D \to \A\) is given by the lax pullback
\[
	\begin{tikzcd}
		\CoCns{D}
				\ar[r] \ar[d] \ar[dr, phantom, "\lrcorner"{description, very near start}]
			&\Terminal[\Cat[\E]]
				\ar[d, "{\Name{D}}"] \\
		\A
				\ar[r, "{\Delta}"']
				\ar[ur, Leftarrow]
			&\PowObj{\D}{\A}
	\end{tikzcd}.
\]

\subsection{Limits}

A universal cone for a given diagram is a terminal object in the category of cones over such diagram.
Though, we have to be careful about which category of cones we are talking about, and what it means to be terminal in it.
Since our definition aims to be internal in essence, what we want is a terminal object, as defined by the internal language, in the internal category of cones.
That is more than a terminal object in the external category of cones: indeed, stability under pullback has to be enforced, thus realizing the intuition that, if \(\lim_D\) is a limit for a diagram \(D \colon \D \to \A\) and \(I\) is an object of \(\E\), then \(I^* \lim_D\) should be a limit for \(I^* D\).

\begin{definition}[universal cone/limit]
	\label{def:internal-limit}
	A cone \((V, \gamma)\) over a diagram \(D \colon \D \to \A\) is \Def{universal} if it is internally a terminal object for \(\Cns{D}\); that means that, for every object \(I\) of \(\E\), the cone \(I^*(V, \gamma)\) over \(I^* D\) is a terminal object in the external category of cones over \(I^* D\).
	The vertex of a universal cone over \(D\) is also called the \Def{limit} of \(D\).
\end{definition}

Dually, we give the definitions of universal cocone and colimit.

\begin{definition}[universal cocone/colimit]
	\label{def:internal-colimit}
	A cocone \((V, \gamma)\) over a diagram \(D \colon \D \to \A\) is \Def{universal} if it is internally an initial object for \(\CoCns{D}\); that means that, for every object \(I\) of \(\E\), the cocone \(I^*(V, \gamma)\) over \(I^* D\) is an initial object in the external category of cocones over \(I^* D\).
	The vertex of a universal cocone over \(D\) is also called the \Def{colimit} of \(D\).
\end{definition}

It is inconvenient that stability needs to be imposed, but such is the price of doing everything concretely, without relying on the internal logic.
However, once the definition has been stabilized, we get all the desired consequences of an internal definition.
For example, the universal property does not merely hold for cones, but for indexed families of cones as well: that is to say, the universal quantifier for all cones has the expected property.

Indeed, let \(D \colon \D \to \A\) be a diagram admitting a universal cone \((\lim_D, \pi)\).
Then, if \(I\) is an object of \(\E\), consider an \(I\)-indexed family of cones over \(D\)
\begin{equation}
\label{eq:I-family-cones}
	\begin{tikzcd}
		\I
				\ar[r] \ar[d, "V"']
			&\Terminal[\Cat[\E]]
				\ar[d, "{\Name{D}}"] \\
		\A
				\ar[r, "{\Delta}"']
				\ar[ur, Rightarrow, "\gamma"]
			&\PowObj{\D}{\A}
	\end{tikzcd}.
\end{equation}
By applying \(I^*\) and the properties of the adjunction, we get a cone over \(I^* D\)
\[
	\begin{tikzcd}
		\Terminal[\Cat[\sfrac{\E}{I}]]
				\ar[r, equal]
				\ar[d, "{\eta(\Terminal[\Cat[\sfrac{\E}{I}]])}"']
			&\Terminal[\Cat[\sfrac{\E}{I}]]
				\ar[d, equal] \\
		I^* \I
				\ar[r] \ar[d, "I^* V"']
			&\Terminal[\Cat[\sfrac{\E}{I}]]
				\ar[d, "{I^* \Name{D}}"]
				\ar[r, equal]
			&\Terminal[\Cat[\sfrac{\E}{I}]]
				\ar[d, "{\Name{I^* D}}"] \\
		I^* \A
				\ar[r, "{I^* \Delta}"']
				\ar[rr, bend right, "{\Delta}"]
				\ar[ur, Rightarrow, "I^* \gamma"]
			&I^* \PowObj{\D}{\A}
				\ar[r, "{i}"']
			&\PowObj{I^* \D}{I^* \A}
	\end{tikzcd}.
\]
Let us call \(\bar{V} = (I^* V) \eta(\Terminal[\Cat[\sfrac{\E}{I}]])\).
Then, by the universal property of \(\lim_D\), there exists a unique \(h \colon \bar{V} \to I^* \lim_D\) such that the above diagram is equal to
\[
	\begin{tikzcd}[row sep = large, column sep = large]
		\Terminal[\Cat[\sfrac{\E}{I}]]
				\ar[r, equal]
					\ar[d, bend right, "{\bar{V}}"' near start, ""{name=s, left}]
					\ar[d, bend left, "{I^* \lim_D}" near start, ""'{name=t, right}]
					\ar[from=s, to=t, Rightarrow, "{h}"]
			&\Terminal[\Cat[\sfrac{\E}{I}]]
				\ar[d, "{I^* \Name{D}}"]
				\ar[r, equal]
			&\Terminal[\Cat[\sfrac{\E}{I}]]
				\ar[d, "{\Name{I^* D}}"] \\
		I^* \A
				\ar[r, "{I^* \Delta}"']
				\ar[rr, bend right, "{\Delta}"]
				\ar[ur, Rightarrow, "I^* \pi"']
			&I^* \PowObj{\D}{\A}
				\ar[r, "{i}"']
			&\PowObj{I^* \D}{I^* \A}
	\end{tikzcd}.
\]
On both squares, we cancel the whiskering with \(i\), which is possible because \(i\) is an isomorphism as \(\E\) is locally cartesian closed.
Then, we apply \(I_!\) and compose with the counit \(\epsilon : I_! I^* \to \Id[\Cat[\E]]\) of the adjunction \(I_! \Adjoint I^*\).
If we let \(\bar{h}\) be \(\epsilon(\A) (I_! h)\), then we get that the square \eqref{eq:I-family-cones} is equal to
\[
	\begin{tikzcd}[row sep = large, column sep = large]
		\I
				\ar[dr, bend right, "{V}"', "{}"{name=s}]
				\ar[r, "{!}"]
			&\Terminal[\Cat[\E]]
				\ar[r, equal]
				\ar[to=s, Leftarrow, "\bar{h}"]
				\ar[d, "{\lim_D}", "{\ }"'{name=t, below}]
			&\Terminal[\Cat[\E]]
				\ar[d, "{\Name{D}}"] \\
		&\A
				\ar[r, "{\Delta}"']
				\ar[ur, Rightarrow, "\pi"']
			&\PowObj{\D}{\A}
	\end{tikzcd}
\]
meaning that even families of cones factorize through the limit.

Moreover, any cone \((V, \gamma)\) over \(D\) corresponds uniquely to a global section \((V, \gamma) \colon \Terminal[\Cat[\E]] \to \Cns{D}\), and a cone is a limit if it is understood in the internal logic as a terminal object for \(\Cns{D}\).
Again, it is essential that the notion is stable under pullback, otherwise it could not be expressed in the internal language.

\subsection{Completeness}

While completeness generally means to ``just'' have limits for all diagrams, in the internal setting the treatment of the notion requires special care.
The category of diagrams of shape \(\D\) over \(\A\) is evidently the functor category \(\PowObj{\D}{\A}\), and the notion of functor category is stable under pullback.
Then, we may wish to express in the internal language the fact that every such diagram has a limit cone.
However, limits are determined only up to (unique) isomorphism, and we have no non-unique existential quantifier in the logic of a cartesian closed category.
Moreover, while we have an internal category of diagrams of a certain shape, there is no internal category of diagrams of all shapes.
Thus, completeness cannot be a purely internal notion.

From an external perspective, it would look natural to call an (internal) category complete if every diagram over it admits a limit.
Again, it is necessary to reconsider this intuition, as the resulting notion would not be stable under pullback.
Indeed, even if every diagram over a category \(\A\) admits a limit, it is still possible that not every diagram over \(I^* \A\) does, for some indexing object \(I\).

\begin{definition}[completeness]\label{def:int-completeness}
	An internal category \(\A\) in \(\E\) is \Def{complete} if, for every object \(I\) of \(\E\), every diagram over \(I^* \A\) admits a limit.
\end{definition}

Such notion of completeness is what in \textcite{Hyland90discreteobjects} is called strong completeness.

The dual notion of cocompleteness is what one would expect.

\begin{definition}[cocompleteness]
	An internal category \(\A\) in \(\E\) is \Def{(strongly) cocomplete} if, for every object \(I\) of \(\E\), every diagram over \(I^* \A\) admits a colimit.
\end{definition}

At this point, it is necessary to find out whether such definition of completeness admits non-trivial examples.
The first place to look into is the category of sets, wherein the internal complete categories are precisely the small complete categories.
That means that the notion of internal completeness is compatible with the standard one, as it should be.

\begin{example}
	\label{example:complete-lattices}
	The internal complete categories in \(\SetCat\) are (up to equivalence) the (categories associated to) complete lattices.
\end{example}

While the examples in the category of sets are legitimate, they are slightly disappointing: there are no small complete categories other than lattices.
Fortunately, and somewhat surprisingly, there are remarkable examples in other ambient categories.
Maybe the most famous of these is the internal category of modest sets in the category of assemblies \autocite{hyland1988small,Hyland90discreteobjects}.

A complete category has all indexed limits, and the object of diagrams of a certain shape over it is a suitable indexing object.
Following this intuition, we can produce a functorial choice of limits in the form of a right adjoint to the diagonal functor.
This shows some interesting facts.
Firstly, that there exists an internal, uniform (as in functorial) choice of limits, as opposed to the external choice required by the definition of internal completeness.
Secondly, that our definition of completeness is consistent with the well-known characterization of completeness via adjunctions.
Finally, that the limit-choice functor comes for free from the definition of completeness, whereas, for standard categories, that would generally require the use of the axiom of choice.
It is worth remarking that the choice functor can only be defined over the object of diagrams of a given shape, as there is no internal object of all diagrams of any shape.

\begin{proposition}
\label{prop:int-limit-funct}
	If \(\A\) is a complete internal category in \(\E\), then for any internal category \(\D\) in \(\E\) there is an internal functor \(\Lim_{\D} \colon \PowObj{\D}{\A} \to \A\) which is right adjoint to the diagonal functor \(\Delta \colon \A \to \PowObj{\D}{\A}\).
\end{proposition}
\begin{proof}
	The first step is to define the functor \(\Lim_{\D}\).
	We shall define the object and arrow components separately, each by using the definition of limit for a suitable diagram, in a suitable slice of \(E\).

	To define the object component of \(\Lim_{\D}\), consider \(\Fun{\D}{\A}\) as the indexing object and the diagram \(\epsilon \colon \Fun{\D}{\A}^*\D \to \Fun{\D}{\A}^*\A\) in \(\Cat[\sfrac{\E}{\Fun{\D}{\A}}]\), defined as
	\begin{align*}
		F \colon \Fun{\D}{\A}, d \colon D_0 &\TiC \epsilon_0(F, d) \DefEq (F, \Eval[\E](F_0, d)) \\
		F \colon \Fun{\D}{\A}, f \colon D_1 &\TiC \epsilon_1(F, f) \DefEq (F, \Eval[\E](F_1, f)).
	\end{align*}
	Then, the diagram admits a universal cone
	\[
		\pi \colon \Lim_\epsilon ! \to \epsilon \colon \Fun{\D}{\A}^*\D \to \Fun{\D}{\A}^*\A
	\]
	with vertex \(\Lim_\epsilon \colon \Terminal[\Cat[\sfrac{\E}{\Fun{\D}{\A}}]] \to \Fun{\D}{\A}^*\A\), which corresponds to an arrow \((\Lim_{\D})_0 \colon \Fun{\D}{\A} \to A_0\) yielding the object component of \(\Lim_{\D}\).

	To define the arrow component of \(\Lim_{\D}\), consider \(\Nat{\D}{\A}\) as the indexing object.
	There are diagrams \(\epsilon_{\Source}, \epsilon_{\Target} \colon \Nat{\D}{\A}^*\D \to \Nat{\D}{\A}^*\A\) in \(\Cat[\sfrac{\E}{\Nat{\D}{\A}}]\), defined as
	\begin{align*}
		\alpha \colon F \to_{\Nat{\D}{\A}} G, d \colon D_0 &\TiC
			{(\epsilon_{\Source})}_0(\alpha, d) \DefEq (\alpha, \Eval[\E] \big(F_0, d) \big) \\
		\alpha \colon F \to_{\Nat{\D}{\A}} G, f \colon D_1 &\TiC
			{(\epsilon_{\Source})}_1(\alpha, f) \DefEq (\alpha, \Eval[\E] \big(F_1, f) \big) \\
	\shortintertext{and}
		\alpha \colon F \to_{\Nat{\D}{\A}} G, d \colon D_0 &\TiC
			{(\epsilon_{\Target})}_0(\alpha, d) \DefEq (\alpha, \Eval[\E] \big(G_0, d) \big) \\
		\alpha \colon F \to_{\Nat{\D}{\A}} G, f \colon D_1 &\TiC
			{(\epsilon_{\Target})}_1(\alpha, f) \DefEq (\alpha, \Eval[\E] \big(G_1, f) \big) \\
	\intertext{and there is also a natural transformation \(\bar{\epsilon} \colon \epsilon_{\Source} \to \epsilon_{\Target}\) defined as}
		\alpha \colon F \to_{\Nat{\D}{\A}} G, d \colon D_0 &\TiC
		\bar{\epsilon}(\alpha, d) \DefEq \big(\alpha, \Eval[\E](\alpha, d) \big).
	\end{align*}
	Then, the diagram \(\epsilon_{\Target}\) admits a universal cone
	\[
		\pi_{\Target} \colon \Lim_{\epsilon_{\Target}} ! \to \epsilon_{\Target} \colon \Nat{\D}{\A}^*\D \to \Nat{\D}{\A}^*\A.
	\]
	with vertex
	\[
		\Lim_{\epsilon_{\Target}} \colon \Terminal[\Cat[\sfrac{\E}{\Nat{\D}{\A}}]] \to \Nat{\D}{\A}^*\A.
	\]
	Consider \(L' \colon \Terminal[\Cat[\sfrac{\E}{\Nat{\D}{\A}}]] \to \Nat{\D}{\A}^*\A\) given, in context \(\alpha \colon F \to_{\Nat{\D}{\A}} G, d \colon D_0\), by \(L'(\alpha) \DefEq (\alpha, (\Lim_{\D})_0(F))\), and the natural transformation \(\pi' \colon L'! \to \epsilon_{\Source}\) defined as
	\[
		\pi'_{\alpha, d} \DefEq
		\big(
			\alpha,
			(\Lim_{\D})_0(F) \xrightarrow{\pi(F, d)} \epsilon(F, d)
		\big).
	\]
	Let \(\pi'' \DefEq \bar{\epsilon} \pi' \colon L'! \to \epsilon_{\Target}\).
	Then, by the universal property of the limit, there is a unique natural transformation
	\[
		(\Lim_{\D})_1 \colon L' \to \Lim_{\epsilon_{\Target}} \colon \Terminal[\Cat[\sfrac{\E}{\Nat{\D}{\A}}]] \to \Nat{\D}{\A}^*\A
	\]
	such that \(\pi'' = \pi_{\Target} (\Lim_{\D})_1\) and whose underlying arrow in \(\E\) is a morphism \((\Lim_{\D})_1 \colon \Nat{\D}{\A} \to A_1\).
	That is the arrow component of \(\Lim_{\D}\).
	
	We shall now give the unit and counit for the adjunction \(\Delta \Adjoint \Lim_{\D}\).
	Define the unit \(\Id(\A) \to \Lim_{\D}\Delta \colon \A \to \A\) as the unique natural isomorphism yielded by the universal property of the limit.
	Define the counit \(\Delta\Lim_{\D} \to \Id(\PowObj{\D}{\A})\) as the natural transformation induced by the universal cone \(\pi\), which can be regarded as an arrow \(\Fun{\D}{\A} \to \Nat{\D}{\A}\).
	That is natural because, in context \((\alpha \colon F \to G) \colon \Nat{\D}{\A}\) and \(d \colon D_0\), the equation \(p'' = p (\Lim_{\D})_1\) implies that the square
	\[
	\begin{tikzcd}
		(\Lim_{\D})_0(F)
				\ar[r, "{\pi(F, d)}"]
				\ar[d, "{(\Lim_{\D})_1(\alpha)}"']
			&\epsilon(F, d) = \Eval[\E](F_0, d)
				\ar[d, "{\Eval[\E](\alpha, d)}"] \\
		(\Lim_{\D})_0(G)
				\ar[r, "{\pi(G, d)}"]
			&\epsilon(G, d) = \Eval[\E](G_0, d)
	\end{tikzcd}
	\]
	commutes.
	These data yield the required adjunction.
\end{proof}

Notice that such limit functor is stable under pullback, as it is the right adjoint of the diagonal functor, which is a stable notion in a locally cartesian closed category.
More precisely, \(I^* \lim_{\D} \Iso \lim_{I^* \D} \colon \A \to \PowObj{\D}{\A}\) for every indexing object \(I\) of \(\E\).

Together, the functor \(\lim_{\D}\) and the counit of the adjunction \(\Delta \lim_{\D} \to \Id[\PowObj{\D}{\A}]\) induce an internal functor \(\PowObj{\D}{\A} \to \Cns{\D}\) which is a right adjoint of the projection \(p \colon \Cns{\D} \to \PowObj{\D}{\A}\).
Likewise, we could produce the functor \(\lim_{\D}\) from the right adjoint of \(p\).
In other words, the two formulations of a limit functor in terms of a right adjoint to \(p \colon \Cns{\D} \to \PowObj{\D}{\A}\) and in terms of a right adjoint to \(\Delta \colon \A \to \PowObj{\D}{\A}\) are equivalent.
Such equivalence is just a matter of routine 2-category-theory calculations.

\subsection{Special Limits}

The strength of the notion of internal completeness is constrained by the shapes of the diagrams we can build in \(\E\).
For example, if \(\E\) has no coproduct \(\Terminal + \Terminal\), we are unable to build a diagram with a pair of objects.
That means that an internal category may in principle be complete, but not have binary products.
Even more strikingly, if \(\E\) does not have an initial object, we are unable to build the degenerate diagram having the terminal object as its limit.
Fortunately, it is possible to express the concepts of terminal object, internal binary product and equalizer independently of the notion of internal limit.

Firstly, we treat the terminal object.

\begin{definition}[internal terminal object]
	An internal category \(\A\) in \(\E\) has terminal object if the functor \(! \colon \A \to \Terminal[\Cat[\E]]\) has a right adjoint.
\end{definition}

\begin{proposition}
	If \(\E\) has initial object and \(\A\) is a complete internal category in \(\E\), then \(\A\) has terminal object.
\end{proposition}
\begin{proof}
	Let \(\Initial[\Cat[\E]]\) be the initial object in \(\Cat[\E]\), built from the initial object \(\Initial\) of \(\E\).
	Then, \(\Delta \colon \A \to \PowObj{\Initial[\Cat[\E]]}{\A}\) is \(! \colon \A \to \Terminal[\Cat[\E]]\), and the limit functor of \(\A\) from \Cref{prop:int-limit-funct} yields the right adjoint.
\end{proof}

Secondly, we turn our attention to binary products.

\begin{definition}[internal binary product]
	An internal category \(\A\) in \(\E\) has binary products if the functor \(\Delta \colon \A \to \A \Prod[\Cat[\E]] \A\) has a right adjoint.
\end{definition}

It turns out that, when one can give a natural definition in terms of diagrams, such definition coincides with the general one.

\begin{proposition}
	If \(\E\) has binary coproducts and \(\A\) is a complete internal category in \(\E\), then \(\A\) has binary products.
\end{proposition}
\begin{proof}
	Consider the internal category \(\DoubleCat = \Terminal[\Cat[\E]] +^{\Cat[\E]} \Terminal[\Cat[\E]]\).
	Then, \(\Delta \colon \A \to \PowObj{\DoubleCat}{\A}\) is \(\Delta \colon \A \to \A \Prod[\Cat[\E]] \A\), and the limit functor of \(\A\) from \Cref{prop:int-limit-funct} yields the right adjoint.
\end{proof}

Finally, we treat the case of equalizers.
Consider the object \({(\Parallel{\A})}_0\) of pairs of parallel arrows in \(\A\), defined by means of the internal language as the subobject of \(A_1 \Prod A_1\) given, in context \((f, g) \colon A_1 \Prod A_1\), by \(\Source[\A](f) = \Source[\A](g)\) and \(\Target[\A](f) = \Target[\A](g)\).
Analogously, there is an object of commutative squares between pairs of parallel arrows in \(\A\), which is the subobject of \({(\Parallel{\A})}_0 \Prod {(\Parallel{\A})}_0 \Prod A_1 \Prod A_1\) given, in context \(( (f, g), (f', g'), h_0, h_1 )\), by
\[
	\begin{tikzcd}[row sep = large]
		X_0
			\ar[r, "{f}", shift left = 2]
			\ar[r, "{g}"', shift right = 2]
			\ar[d, "{h_0}"']
		&X_1
			\ar[d, "{h_1}"] \\
		Y_0
			\ar[r, shift left = 2, "{f'}"]
			\ar[r, shift right = 2, "{g'}"']
		&Y_1
	\end{tikzcd}
\]
where \(h_1 \Comp[\A] f = f' \Comp[\A] h_0\) and \(h_1 \Comp[\A] g = g' \Comp[\A] h_0\).
Thus we get the category \(\Parallel{\A}\) of parallel arrows of \(\A\).
There is also a delta functor \(\Delta \colon \A \to \Parallel{\A}\) sending an object of \(\A\) into the pair given by its identity arrow (twice).

\begin{definition}[internal equalizer]
	An internal category \(\A\) in \(\E\) has equalizers if the functor \(\Delta \colon \A \to \Parallel{\A}\) has a right adjoint.
\end{definition}

Again, when we can give a definition in terms of diagrams which coincides with the general definition.

\begin{proposition}
	If \(\E\) has binary coproducts and \(\A\) is a complete internal category in \(\E\), then \(\A\) has equalizers.
\end{proposition}
\begin{proof}
	Let \(\ParallelCat\) be the internal category in \(\E\) with two objects and two parallel arrows between them.
	Then, \(\Delta \colon \A \to \PowObj{\ParallelCat}{\A}\) is \(\Delta \colon \A \to \Parallel{\A}\), and the limit functor of \(\A\) from \Cref{prop:int-limit-funct} yields the right adjoint.
\end{proof}

Notice that the previous notions are all stable under pullback, as they are defined as the right adjoint of the diagonal functor, which is a stable notion in a locally cartesian closed category.
Moreover, we could similarly define the dual special colimits: initial object, binary sum and coequalizer.

\subsection{Completeness and Cocompleteness}

It is part of the folklore of category theory that, were it not for issues of size, completeness and cocompleteness would be equivalent notions.
For example, this idea plays a manifest role in the proof of the General Adjoint Functor Theorem.
In the context of internal categories, the size issues disappear and there is an equivalence.
This important fact is alluded to in \textcite{hyland1988small}, but there is no written account of which we are aware.
Hence, we think it useful to provide a detailed proof in a modern categorical style.

The proof of the above fact is a consequence of the following propositions.

\begin{proposition}
	\label{prop:initial-limit-identity}
	Let \(\A\) be a category in \(\E\).
	If the identity \(\Id[\Cat[\E]](\A) \colon \A \to \A\) has a limit, then that limit is an initial object in \(\A\).
\end{proposition}
\begin{proof}
	Let the functor \(I \colon \Terminal[\Cat[\E]] \to \A\) be the limit of \(\Id[\Cat[\E]](\A)\), and the natural transformation \(\pi \colon I !_{\A} \to \Id[\Cat[\E]](\A) \colon \A \to \A\) be its universal cone.
	We want to prove that \(I\) is the left adjoint of \(!_{\A}\), with \(\pi\) as the counit of the adjunction (the unit being the obvious degenerate natural transformation).
	Then, the only non-trivial triangular equation (after canceling the identity terms) requires us to prove that \(\pi I = \Id(I)\).
	Observe that the interchange law of composition applied to the natural transformations
	\[
		\begin{tikzcd}[column sep = small]
			&|[alias=Z]| \Terminal[\Cat[\E]]
				\ar[dr, "{I}"] \\
			\A
				\ar[ur, "!"]
				\ar[rr, equal, ""{name=C, above}]
				\ar[from=Z, to=C, Rightarrow, "{\pi}" description]
			&{\ }
			&\A
				\ar[dr, "!"']
				\ar[rr, equal, ""{name=D, below}]
			&{\ }
			&\A \\
			&&&\Terminal[\Cat[\E]]
				\ar[ur, "{I}"']
				\ar[to=D, Rightarrow, "{\pi}" description]
		\end{tikzcd}
	\]
	says that \(\pi = \pi ( \pi I! )\) (given that \(I! \pi\) is an identity as it factors through \(\Terminal[\Cat[\E]]\)).
	Then, by the universal property of the limit, there is a unique \(h \colon I \to I\) such that \(\pi (h !) = \pi\), which is obviously \(\Id[I]\).
	But, by the previous observation, also \(\pi I\) features such property, and then it must be that \(\pi I = \Id[I]\), which is the required triangular equation.
\end{proof}

\begin{proposition}
	\label{prop:cocone-complete}
	Let \(\A\) be a complete category in \(\E\) and \(D \colon \D \to \A\) a diagram over it.
	Then, \(\CoCns{D}\) is complete too.
\end{proposition}
\begin{proof}
	Let \(D' \colon \D' \to \CoCns{D}\) be a diagram over \(\CoCns{D}\).
	Observe	that, by composition with the projection \(T \colon \CoCns{D} \to \A\), we have a diagram \(T D'\) over \(\A\).
	By hypothesis, \(T D'\) has a limit \(\lim_{T D'} \colon \Terminal[\Cat[\E]] \to \A\) with universal cone \(\pi \colon \Delta \lim_{T D'} \to \Name{T D'}\).
	The situation is as follows.
	\[
	\begin{tikzcd}
			\D'
				\ar[r, "{D'}"]
				\ar[d]
			&\CoCns{D}
				\ar[r] \ar[d, "{T}"']
				\ar[dr, phantom, "\lrcorner"{description, very near start}]
			&\Terminal[\Cat[\E]]
				\ar[d, "{\Name{D}}"] \\
			\Terminal[\Cat[\E]]
				\ar[r, "{\lim_{T D'}}"', {name=L}]
				\ar[ur, Rightarrow, "{\pi}" description]
			&\A
				\ar[r, "{\Delta}"']
				\ar[ur, Leftarrow, "{\gamma}" description]
			&\PowObj{\D}{\A}
	\end{tikzcd}
	\]
	Because \(\Delta\) preserves limits, \(\lim_{\Delta T D'}\) is (isomorphic to) \(\Delta \lim_{T D'}\) and its universal cone is (isomorphic to) \(\Delta \pi\).
	But \(\gamma D' \colon \Name{D}! \to \Delta T D'\) is a cone over \(\Delta T D'\) with vertex \(\Name{D}\), so, by the universal property of the limit, there is a unique natural transformation \(h \colon \Name{D} \to \Delta \lim_{T D'}\) such that \((\Delta \pi) ( h ! ) = \gamma D'\).
	Then, by the universal property of the lax pullback, there is a unique \(L \colon \Terminal[\Cat[\E]] \to \CoCns{D}\) such that \(T L \Iso \lim_{T D'}\) and (up to this isomorphism) \(\gamma L = h\).
	Moreover, observe that by the defining properties of \(h\) and \(L\) it holds \(\gamma D' = (\Delta \pi) ( \gamma L! )\).
	Then, by the 2-dimensional universal property of the lax pullback, there is a unique \(p \colon L ! \to D'\) such that \(T p = \pi\).
	
	We claim \((L, p)\) is a limit for \(D'\).
	We need to show that, for any other cone \((L', p')\) on \(D'\), there is a unique \(\beta \colon L' \to L\) such that \(p (\beta !) = p'\).
	
	Let's produce \(\beta\).
	To begin with, notice that \(T p' \colon T L' ! \to T D'\) is a cone on \(T D'\) and thus, by the universal property of the limit, there exists a unique \(\kappa \colon T L' \to \lim_{T D'}\) such that \(\pi ( \kappa ! ) = T p'\).
	Then, observe that \((\Delta \kappa) (\gamma L')\) is \(h\), as it has the same unique property:
	\[
		\begin{split}
			\MoveEqLeft (\Delta \pi) \Big( \big( (\Delta \kappa) (\gamma L') \big) ! \Big) \\
			&= (\Delta \pi) (\Delta \kappa !) (\gamma L'!) \\
			&= \Big( \Delta \big( \pi (\kappa !) \big) \Big) (\gamma L'!) \\
		\text{(by the unique property of \(\kappa\))}
			&= ( \Delta T p' ) ( \gamma L'! ) \\
		\text{(by the interchange law of composition)}
			&= ( \gamma D' ) ( \Name{D} ! p' ) \\
		\text{(because \(\Name{D} ! p'\) factors through \(\Terminal[\Cat[\E]]\))}
			&= \gamma D'.
		\end{split}
	\]
	Finally, by the 2-dimensional universal property of the lax pullback, there is a unique \(\beta \colon L' \to L\) such that \(T \beta = \kappa\).
	
	Now we need to show that \(p'\) factorizes through \(p\), i.e.\ that \(p (\beta !) = p'\).
	Notice that \(\gamma D' = ( \Delta ( \pi ( \kappa ! ) ) ) (\gamma L'!)\) (it is the same calculation as before).
	Then, there is a unique \(q \colon L'! \to D'\) such that \(T q = \pi ( \kappa ! )\).
	But both \(p (\beta !)\) and \(p'\) have this unique property.
	Indeed, \(T ( p (\beta !) ) = (T p) (T \beta !)\) and that, because of the unique properties of \(p\) and \(\beta\), is \(\pi (\kappa !)\); and for \(p'\) it holds because the unique property of \(\kappa\).
	But then, \(p (\beta !) = p'\).
	
	To conclude, let's prove that \(\beta\) is unique with its property.
	Let \(\beta' \colon L' \to L\) be such that \(p (\beta' !) = p'\).
	To prove that \(\beta'\) is \(\beta\), we need to show that it has the same unique property, i.e.\ that \(T \beta' = \kappa\).
	Likewise, to prove that \(T \beta'\) is \(\kappa\), we need show that it has the same unique property, i.e.\ that \(\pi ( T \beta' ! ) = T p'\).
	But that follows from \((T p) (T \beta' !) = T p'\), by applying the unique property of \(p\).
\end{proof}

We can now finally prove the result.

\begin{theorem}
\label{thm:complete-cocomplete}
	Let \(\A\) be a category in \(\E\).
	Then, \(\A\) is complete if and only if it is cocomplete.
\end{theorem}
\begin{proof}
	By duality, it is enough to show only one direction of the equivalence.

	Assume \(\A\) is complete, and let's prove that it is also cocomplete.
	Consider the diagram \(D \colon \D \to \A\).
	A colimit of \(D\) is by definition an initial object in \(\CoCns{D}\), which is, by \Cref{prop:initial-limit-identity}, the identity on \(\CoCns{D}\).
	Such limit exists because, by \Cref{prop:cocone-complete}, \(\CoCns{D}\) is complete.
\end{proof}

The equivalence of completeness and cocompleteness is, essentially, a generalization of the well-known fact that complete lattices are also cocomplete.
We can now retrieve such a fact by recalling that complete lattices are complete internal categories in \(\SetCat\) (\Cref{example:complete-lattices}) and by applying \Cref{thm:complete-cocomplete}.

Another immediate consequence is that the category of modest sets in the category of assemblies is cocomplete.

\subsection{Adjoint Functor Theorem}

The adjoint functor theorem is a fundamental result of category theory, dating back to \textcite{freyd1964abelian} (appearing in the exercise section of Chapter 3), which states that, under suitable conditions, a limit-preserving functor has a right adjoint.
More appropriately, we should talk about the family of adjoint functor theorems, as there are many alternative such conditions that, fundamentally, deal with completeness and size issues.
We shall show that many complications fade away in the internal context, where we only need the domain of the functor to be a complete category.

Let \(F \colon \A \to \B\) be a functor between internal categories in \(\E\).
Observe that there is a post-composition functor \(\PowObj{\D}{F} \colon \PowObj{\D}{\A} \to \PowObj{\D}{\B}\)
such that, for any diagram \(D \colon \D \to \A\), we have that \(\PowObj{\D}{F} \Name{D} = \Name{F D}\)
and \(\PowObj{\D}{F} \Delta = \Delta F\).

We shall define a notion of continuity for internal functors.

\begin{definition}[continuous functor]
	Let \(F \colon \A \to \B\) be a functor between internal categories in \(\E\).
	We say \(F\) is \Def{continuous} if, for any object \(I\) and every diagram \(D \colon \D \to I^*\A\) admitting a universal cone
	\(\pi \colon \Delta \lim_D \to \Name{D}\), then \(\PowObj{\D}{F} \pi \colon \Delta F \lim_D \to \Name{F D}\)
	is a universal cone for \(FD\), and thus \(F \lim_D\) is a limit for \(FD\).
\end{definition}

We choose to name such functors as ``continuous'' rather then, more descriptively, as ``limit-preserving,'' since, in the internal context, all diagrams are small.

We can now prove an adjoint functor theorem for internal categories.

\begin{theorem}[adjoint functor theorem]
	Let \(\B\) be a complete internal category in \(\E\) and \(R \colon \B \to \A\) be a continuous functor between internal categories in \(\E\).
	Then, \(R\) is a right adjoint, i.e., it has a left adjoint \(L \colon \A \to \B\).
\end{theorem}
\begin{proof}
We split the proof into smaller steps:
\begin{itemize}
	\item We shall define the functor \(L \colon \A \to \B\).
	
	Let \(\sfrac{\A}{R}\) be the comma category given by the lax pullback
	\[
		\begin{tikzcd}
			\sfrac{\A}{R} \ar[r, "{P_{\A}}"] \ar[d, "{P_{\B}}"']
			&\A \ar[d, equal] \ar[dl, Rightarrow, "{\alpha}"] \\
			\B \ar[r, "{R}"']
			&\A
		\end{tikzcd}.
	\]
	Consider the diagram
	\[
		\begin{tikzcd}
			\sfrac{\A}{R} \ar[rr, "{(P_{\A}, P_{\B})}"] \ar[dr, "{P_{\A}}"']
			&&\A \Prod \B \ar[dl, "{\pi_{1}}"] \\
			&\A
		\end{tikzcd}
	\]
	in \(\sfrac{\Cat[\E]}{\A}\).
	By completeness of \(\B\), the diagram must have a limit
	\[
		\begin{tikzcd}
			\A \ar[rr, "{\lim_{(P_{\A}, P_{\B})}}"] \ar[dr, equal]
			&&\A \Prod \B \ar[dl, "{\pi_{1}}"] \\
			&\A
		\end{tikzcd}
	\]
	with universal cone \(\pi \colon \Delta \lim_{(P_{\A}, P_{\B})} \to \Name{(P_{\A}, P_{\B})}\),
	where the situation is as shown in the following diagram.
	\[
		\begin{tikzcd}[column sep = large, row sep = huge]
			\A
				\ar[rr, bend left = 15, "{\Delta \lim_{(P_{\A}, P_{\B})}}" description]
				\ar[rr, bend right = 15, "{\Name{(P_{\A}, P_{\B})}}"' description]
				\ar[dr, equal]
			&&\PowObj{\sfrac{\A}{R}}{\A \Prod \B}_{\A} \ar[dl, "{\PowObj{P_{\A}}{\pi_{1}}_{\A}}"] \\
			&\A
		\end{tikzcd}
	\]
	The universal cone \(\pi\) corresponds, by currying, to a natural transformation
	\(\tilde{\pi} \colon \lim_{(P_{\A}, P_{\B})} !_{\sfrac{\A}{R}} \to (P_{\A}, P_{\B})\),
	where the situation is as shown in the following diagram.
	\[
		\begin{tikzcd}[column sep = large, row sep = huge]
			\sfrac{\A}{R}
				\ar[rr, bend left = 15, "{\lim_{(P_{\A}, P_{\B})} !_{\sfrac{\A}{R}}}"]
				\ar[rr, bend right = 15, "{(P_{\A}, P_{\B})}"']
				\ar[dr, "{P_{\A}}"']
			&& \A \Prod \B
				\ar[dl, "{\pi_{1}}"] \\
			&\A
		\end{tikzcd}
	\]
	It follows from the above diagrams that there exist a functor \(L \colon \A \to \B\) such that \(\lim_{(P_{\A}, P_{\B})} = (\Id[\A], L)\),
	and a natural transformation \(\hat{\pi} \colon L P_{\A} \to P_{\B}\) such that \(\tilde{\pi} = (\Id[P_{\A}], \hat{\pi})\).

	\item We shall define the natural transformation \(\eta \colon \Id[\A] \to R L\).
	
	Since \(R\) is continuous, \((\Id[\A] \Prod R) \lim_{(P_{\A}, P_{\B})}\) is a limit for \((P_{\A}, R P_{\B})\).
	Notice that \((\Id[\A] \Prod R) \lim_{(P_{\A}, P_{\B})} = (\Id[\A] \Prod R) (\Id[\A], L) = (\Id[\A], RL)\),
	as shown in the following diagram.
	\[
		\begin{tikzcd}[column sep = huge]
			\A \ar[r, "{(\Id[\A], L)}"] \ar[rr, bend left, "{(\Id[\A], RL)}"] \ar[dr, equal]
			&\A \Prod \B \ar[d, "{\pi_{1}}"] \ar[r, "{\Id[\A] \Prod R}"]
			&\A \Prod \A \ar[dl, "{\pi_{1}}"] \\
			&\A
		\end{tikzcd}
	\]
	Notice that \(\PowObj{\sfrac{\A}{R}}{(\Id[\A] \Prod R)} \Delta (\Id[\A], L) = \Delta (\Id[\A], RL)\)
	and that \(\PowObj{\sfrac{\A}{R}}{(\Id[\A] \Prod R)} \Name{(P_{\A}, P_{\B})} = \Name{(P_{\A}, R P_{\B})}\).
	Then, the universal cone for the limit is
	\[
		\PowObj{\sfrac{\A}{R}}{(\Id[\A] \Prod R)} \pi
		\colon \Delta (\Id[\A] \Prod RL) \to \Name{(P_{\A}, R P_{\B})}.
	\]
	The situation is shown in the diagram
	\[
		\begin{tikzcd}[column sep = large, row sep = huge]
			\A
				\ar[rr, bend left = 15, "{\Delta (\Id[\A], RL)}" description]
				\ar[rr, bend right = 15, "{\Name{(P_{\A}, R P_{\B})}}" description]
				\ar[dr, equal]
			&&\PowObj{\sfrac{\A}{R}}{\A \Prod \A}_{\A} \ar[dl, "{\PowObj{P_{\A}}{\pi_{1}}_{\A}}"] \\
			&\A
		\end{tikzcd}
	\]
	where \(\PowObj{P_{\A}}{\pi_{1}}_{\A} \colon \PowObj{\sfrac{\A}{R}}{\A \Prod \A}_{\A} \to \A\) is the exponential of \(P_{\A}\) and \(\pi_{1}\) in \(\sfrac{\Cat[\E]}{\A}\).

	Consider now the cone given by the functor
	\[
		\begin{tikzcd}[column sep = huge]
			\A \ar[rr, "{\Delta_{\A}}"] \ar[dr, equal]
			&&\A \Prod \A \ar[dl, "{\pi_{1}}"] \\
			&\A
		\end{tikzcd}
	\]
	and the natural transformation \(\hat{\alpha} \colon \Delta \Delta_{\A} \to \Name{(P_{\A}, R P_{\B})}\) obtained by uncurrying the natural transformation
	\((\Id[P_{\A}], \alpha) \colon \Delta_{\A} P_{\A} \to (P_{\A}, R P_{\B})\).
	The situation is as shown in the following diagrams.
	\[
		\begin{tikzcd}[column sep = large, row sep = huge]
			\sfrac{\A}{R}
				\ar[rr, bend left = 15, "{\Delta_{\A} P_{\A}}"]
				\ar[rr, bend right = 15, "{(P_{\A}, R P_{\B})}"']
				\ar[dr, "{P_{\A}}"']
			&& \A \Prod \A \ar[dl, "{\pi_{1}}"] \\
			&\A
		\end{tikzcd}
		\begin{tikzcd}[column sep = large, row sep = huge]
			\A
				\ar[rr, bend left = 15, "{\Delta \Delta_{\A}}" description]
				\ar[rr, bend right = 15, "{\Name{(P_{\A}, R P_{\B})}}" description]
				\ar[dr, equal]
			&&\PowObj{\sfrac{\A}{R}}{\A \Prod \A}_{\A} \ar[dl, "{\PowObj{P_{\A}}{\pi_{1}}_{\A}}"] \\
			&\A
		\end{tikzcd}
	\]

	Then, there exists a unique \((\Id[\A], \eta) \colon (\Id[\A], \Id[\A]) \to (\Id[\A], RL)\) such that
	\(\hat{\alpha} = (\Id[\A], \eta) \PowObj{\sfrac{\A}{R}}{(\Id[\A] \Prod R)} \pi\),
	i.e., there exists a unique \(\eta \colon \Id[\A] \to RL\) such that the following equation holds.
	\begin{equation}\label{eq:univ-prop-eta}
		\alpha = (R \hat{\pi}) (\eta P_{\A})
	\end{equation}
	
	\item We shall define the natural transformation \(\epsilon \colon L R \to \Id[\B]\).
	
	Consider the natural transformation \(\Id[R] \colon R \to R\) in the diagram
	\[
		\begin{tikzcd}
			\B
				\ar[ddr, bend right, equal]
				\ar[drr, bend left, "{R}"] \\
			&\sfrac{\A}{R} \ar[r, "{P_{\A}}"] \ar[d, "{P_{\B}}"']
			&\A \ar[d, equal] \ar[dl, Rightarrow, "{\alpha}"] \\
			&\B \ar[r, "{R}"']
			&\A
		\end{tikzcd}.
	\]
	Then, by the defining property of the lax pullback,
	there exists \(I \colon \B \to \sfrac{\A}{R}\)
	and isomorphisms \(\beta_{\A} \colon P_{\A} I \Iso R\)
	and \(\beta_{\B} \colon P_{\B} I \Iso \Id[\B]\)
	such that the following equation holds.
	\begin{equation}\label{eq:def-prop-I}
		(R \beta_{\B}) (\alpha I) \beta_{\A}^{-1} = \Id[R]
	\end{equation}
	
	Let let \(\epsilon\) be defined as
	\begin{equation}\label{eq:def-epsilon}
		\epsilon = \beta_{\B} (\hat{\pi} I) (L \beta_{\A}^{-1}). %
	\end{equation}
	
	\item We shall prove the first triangular identity, i.e., that \((\epsilon L)(L \eta) = \Id[L]\).
	
	We shall show that \(\big( \Id[\Id[\A]], (\epsilon L)(L \eta) \big)\)
	is a morphism of cones \(\pi \to \pi\),
	so that the result will follow by uniqueness of the universal property.
	That amounts to show that
	\(
		\pi \Delta \big( \Id[\Id[\A]], (\epsilon L)(L \eta) \big) = \pi
	\)
	or, equivalently, that
	\begin{equation}\label{eq:first-triang-id}
		\hat{\pi} (\epsilon L P_{\A})(L \eta P_{\A}) = \hat{\pi}.
	\end{equation}

	Consider the natural transformation \(R \hat{\pi} \colon R L P_{\A} \to R P_{\B}\) in the diagram
	\[
		\begin{tikzcd}
			\sfrac{\A}{R} \ar[r, "{P_{\A}}"] \ar[ddr, bend right, "{P_{\B}}"']
			&\A \ar[r, "{L}"]
			&\B \ar[d, "{R}"] \\
			&\sfrac{\A}{R} \ar[r, "{P_{\A}}"] \ar[d, "{P_{\B}}"']
			&\A \ar[d, equal] \ar[dl, Rightarrow, "{\alpha}"] \\
			&\B \ar[r, "{R}"']
			&\A
		\end{tikzcd}.
	\]
	Then, by the defining property of the lax pullback,
	there exists an arrow \(N \colon \sfrac{\A}{R} \to \sfrac{\A}{R}\)
	and isomorphisms \(\beta'_{\A} \colon P_{\A} N \Iso RLP_{\A}\)
	and \(\beta'_{\B} \colon P_{\B} N \Iso P_{\B}\)
	such that
	\begin{equation}\label{eq:def-prop-N}
		R \hat{\pi} = (R \beta'_{\B}) (\alpha N) (\beta'_{\A})^{-1}.
	\end{equation}
	
	Consider the functor \(I L P_{\A} \colon \sfrac{\A}{R} \to \sfrac{\A}{R}\) and
	the natural transformations \((\beta'_{\A})^{-1} (\beta_{\A} L P_{\A}) \colon P_{\A} I L P_{\A} \to P_{\A} N\)
	and \((\beta'_{\B})^{-1} \hat{\pi} (\beta_{\B} L P_{\A}) \colon P_{\B} I L P_{\A} \to P_{\B} N\).
	Then, by the 2-universal property of the lax pullback, there is a (unique) natural transformation \(\omega \colon I L P_{\A} \to N\) such that
	\begin{align}
		P_{\A} \omega &= (\beta'_{\A})^{-1} (\beta_{\A} L P_{\A}) \label{eq:PA-omega} \\
		P_{\B} \omega &= (\beta'_{\B})^{-1} \hat{\pi} (\beta_{\B} L P_{\A}) \label{eq:PB-omega}
	\end{align}
	since
	\[
	\begin{split}
		\MoveEqLeft (R ((\beta'_{\B})^{-1} \hat{\pi} (\beta_{\B} L P_{\A}))) (\alpha I L P_{\A}) \\
		&= (R\beta'_{\B})^{-1} (R \hat{\pi}) (R \beta_{\B} L P_{\A}) (\alpha I L P_{\A}) \\
	\text{(by \Cref{eq:def-prop-N})}
		&= (\alpha N) (\beta'_{\A})^{-1} (R \beta_{\B} L P_{\A}) (\alpha I L P_{\A}) \\
		&= (\alpha N) (\beta'_{\A})^{-1} (((R \beta_{\B}) (\alpha I)) L P_{\A}) \\
	\text{(by \Cref{eq:def-prop-I})}
		&= (\alpha N) (\beta'_{\A})^{-1} (\beta_{\A} L P_{\A}).
	\end{split}
	\]
	By exchange law, we have that \((P_{\B} \omega) (\hat{\pi} I L P_{\A}) = (\hat{\pi} N) (L P_{\A} \omega)\),
	i.e., by replacing \(P_{\A} \omega\) and \(P_{\B} \omega\), that
	\((\beta'_{\B})^{-1} \hat{\pi} (\beta_{\B} L P_{\A}) (\hat{\pi} I L P_{\A}) = (\hat{\pi} N) (L ((\beta'_{\A})^{-1} (\beta_{\A} L P_{\A})))\),
	which can be rewritten as
	\begin{equation}\label{eq:exch-omega}
		\hat{\pi} (\beta_{\B} L P_{\A}) (\hat{\pi} I L P_{\A}) = \beta'_{\B} (\hat{\pi} N) (L \beta'_{\A})^{-1} (L \beta_{\A} L P_{\A}).
	\end{equation}

	We can then proceed to the first part of the calculation to prove \Cref{eq:first-triang-id}.
	Observe that
	\[
	\begin{split}
		\MoveEqLeft \hat{\pi} (\epsilon L P_{\A}) \\
	\text{(by \Cref{eq:def-epsilon})}
		&= \hat{\pi} \big( (\beta_{\B} (\hat{\pi} I) (L \beta_{\A})^{-1}) L P_{\A} \big) \\
		&= \hat{\pi} (\beta_{\B} L P_{\A}) (\hat{\pi} I L P_{\A}) (L \beta_{\A} L P_{\A})^{-1} \\
	\text{(by \Cref{eq:exch-omega})}
		&= \beta'_{\B} (\hat{\pi} N) (L \beta'_{\A})^{-1} (L \beta_{\A} L P_{\A}) (L \beta_{\A} L P_{\A})^{-1} \\
		&= \beta'_{\B} (\hat{\pi} N) (L \beta'_{\A})^{-1}
	\end{split}
	\]
	from which it follows that
	\begin{equation}
		\hat{\pi} (\epsilon L P_{\A}) (L \eta  P_{\A})
		= \beta'_{\B} (\hat{\pi} N) (L \beta'_{\A})^{-1} (L \eta P_{\A}).
	\end{equation}
		
	Consider the natural transformation \(\eta \colon \Id[\A] \to R L\) in the diagram
	\[
		\begin{tikzcd}
			\A \ar[rrd, bend left, equal] \ar[ddr, bend right, "{L}"'] \\ %
			&\sfrac{\A}{R} \ar[r, "{P_{\A}}"] \ar[d, "{P_{\B}}"']
			&\A \ar[d, equal] \ar[dl, Rightarrow, "{\alpha}"] \\
			&\B \ar[r, "{R}"']
			&\A
		\end{tikzcd}.
	\]
	Then, by the defining property of the lax pullback,
	there exists an arrow \(E \colon \A \to \sfrac{\A}{R}\)
	and isomorphisms \(\beta''_{\A} \colon P_{\A} E \Iso \Id[\A]\)
	and \(\beta''_{\B} \colon P_{\B} E \Iso L\)
	such that
	\begin{equation}\label{eq:def-prop-E}
		\eta = (R \beta''_{\B}) (\alpha E) (\beta''_{\A})^{-1}.
	\end{equation}
	
	Consider the functor \(IL \colon \A \to \sfrac{\A}{R}\)
	and the natural transformations \((\beta_{\A} L)^{-1} \eta \beta''_{\A} \colon P_{\A} E \to P_{\A} I L\)
	and \((\beta_{\B} L)^{-1} \beta''_{\B} \colon P_{\B} E \to P_{\B} I L\).
	Than, by the 2-universal property of the lax pullback, there is a unique natural transformation \(\omega' \colon E \to I L\) such that
	\begin{align}
		P_{\A} \omega' &= (\beta_{\A} L)^{-1} \eta \beta''_{\A} \label{eq:PA-omega'}\\
		P_{\B} \omega' &= (\beta_{\B} L)^{-1} \beta''_{\B} \label{eq:PB-omega'}
	\end{align}
	since
	\[
	\begin{split}
		(R ((\beta_{\B} L)^{-1} \beta''_{\B})) (\alpha E)
		&= (R \beta_{\B} L)^{-1} (R \beta''_{\B}) (\alpha E) \\
	\text{(by \Cref{eq:def-prop-E})}
		&= ((R \beta_{\B})^{-1} L) \eta \beta''_{\B} \\
	\text{(by \Cref{eq:def-prop-I})}
		&= (((\alpha I) \beta_{\A}^{-1}) L) \eta \beta''_{\B} \\
		&= (\alpha I L) (\beta_{\A} L)^{-1} \eta \beta''_{\A}.
	\end{split}
	\]

	Then, we can resume the calculation to prove \Cref{eq:first-triang-id}:
	\[
	\begin{split}
		\MoveEqLeft \beta'_{\B} (\hat{\pi} N) (L ((\beta'_{\A})^{-1} (\eta P_{\A})) \\
	\text{(by \Cref{eq:PA-omega'})}
		&= \beta'_{\B} (\hat{\pi} N) \big( L \big( (\beta'_{\A})^{-1} (\beta_{\A} L P_{\A}) (P_{\A} \omega' P_{\A}) (\beta''_{\A} P_{\A})^{-1} \big) \big) \\
	\text{(by \Cref{eq:PA-omega})}
		&= \beta'_{\B} (\hat{\pi} N) (L ((P_{\A} \omega) (P_{\A} \omega' P_{\A}) (\beta''_{\A} P_{\A})^{-1} )) \\
		&= \beta'_{\B} (\hat{\pi} N) (LP_{\A} (\omega (\omega' P_{\A}))) (L \beta''_{\A} P_{\A})^{-1} \\
	\text{(by exchange law)}
		&= \beta'_{\B} (P_{\B} (\omega (\omega' P_{\A}))) (\hat{\pi} E P_{\A}) (L \beta''_{\A} P_{\A})^{-1} \\
		&= \beta'_{\B} (P_{\B} \omega) (P_{\B} \omega' P_{\A}) (\hat{\pi} E P_{\A}) (L \beta''_{\A} P_{\A})^{-1} \\
	\text{(by \Cref{eq:PB-omega})}
		&= \beta'_{\B} (\beta'_{\B})^{-1} \hat{\pi} (\beta_{\B} L P_{\A}) (P_{\B} \omega' P_{\A}) (\hat{\pi} E P_{\A}) (L \beta''_{\A} P_{\A})^{-1} \\
	\text{(by \Cref{eq:PB-omega'})}
		&= \hat{\pi} (\beta_{\B} L P_{\A}) (((\beta_{\B} L)^{-1} \beta''_{\B}) P_{\A}) (\hat{\pi} E P_{\A}) (L \beta''_{\A} P_{\A})^{-1} \\
		&= \hat{\pi} (\beta_{\B} L P_{\A}) (\beta_{\B} L P_{\A})^{-1} (\beta''_{\B} P_{\A}) (\hat{\pi} E P_{\A}) (L \beta''_{\A} P_{\A})^{-1} \\
		&= \hat{\pi} (\beta''_{\B} P_{\A}) (\hat{\pi} E P_{\A}) (L \beta''_{\A} P_{\A})^{-1}
	\end{split}
	\]
	
	Consider the functors \(E P_{\A} \colon \sfrac{\A}{R} \to \sfrac{\A}{R}\) and \(\Id[\sfrac{\A}{R}]\)
	and the natural transformations \(\beta''_{\A} P_{\A} \colon P_{\A} E P_{\A} \to P_{\A}\)
	and \(\hat{\pi} (\beta''_{\B} P_{\A}) \colon P_{\B} E P_{\A} \to P_{\B}\).
	Then, by the 2-universal property of the lax pullback, there is a unique natural transformation \(\omega'' \colon E P_{\A} \to \Id[\sfrac{\A}{R}]\) such that
	\begin{align}
		P_{\A} \omega'' &= \beta''_{\A} P_{\A} \label{PA-omega''} \\
		P_{\B} \omega'' &= \hat{\pi} (\beta''_{\B} P_{\A}) \label{PB-omega''}
	\end{align}
	since
	\[
	\begin{split}
		(R (\hat{\pi} (\beta''_{\B} P_{\A}))) (\alpha E P_{\A})
		&= (R \hat{\pi}) (R \beta''_{\B} P_{\A}) (\alpha E P_{\A}) \\
		&= (R \hat{\pi}) (((R \beta''_{\B}) (\alpha E)) P_{\A}) \\
	\text{(by \Cref{eq:def-prop-E})}
		&= (R \hat{\pi}) ((\eta \beta_{\A}'') P_{\A}) \\
		&= (R \hat{\pi}) (\eta P_{\A}) (\beta_{\A}'' P_{\A}) \\
	\text{(by \Cref{eq:univ-prop-eta})}
		&= \alpha (\beta''_{\A} P_{\A}).
	\end{split}
	\]
	
	Then, we can conclude the calculation to prove \Cref{eq:first-triang-id}:
	\[
	\begin{split}
		\MoveEqLeft \hat{\pi} (\beta''_{\B} P_{\A}) (\hat{\pi} E P_{\A}) (L \beta''_{\A} P_{\A})^{-1} \\
	\text{(by \Cref{PB-omega''})}
		&= (P_{\B} \omega'') (\hat{\pi} E P_{\A}) (L \beta''_{\A} P_{\A})^{-1} \\
	\text{(by exchange law)}
		&= \hat{\pi} (L P_{\A} \omega'') (L \beta''_{\A} P_{\A})^{-1} \\
	\text{(by \Cref{PA-omega''})}
		&= \hat{\pi} (L \beta''_{\A} P_{\A}) (L \beta''_{\A} P_{\A})^{-1} \\
		&= \hat{\pi}.
	\end{split}
	\]
	
	\item Finally, we shall prove the second triangular identity: %
	\[
	\begin{split}
		\MoveEqLeft (R \epsilon)(\eta R) \\
	\text{(by \Cref{eq:def-epsilon})}
		&= (R (\beta_{\B} (\hat{\pi} I) (L \beta_{\A}^{-1}))) (\eta R) \\
		&= (R \beta_{\B}) (R \hat{\pi} I) (R L \beta_{\A}^{-1}) (\eta R) \\
	\text{(by exchange law)}
		&= (R \beta_{\B}) (R \hat{\pi} I) (\eta P_{\A} I) \beta_{\A}^{-1} \\
		&= (R \beta_{\B}) (((R \hat{\pi}) (\eta P_{\A})) I ) \beta_{\A}^{-1} \\
	\text{(by \Cref{eq:univ-prop-eta})}
		&= (R \beta_{\B}) (\alpha I) \beta_{\A}^{-1} \\
	\text{(by \Cref{eq:def-prop-I})}
		&= \Id[R].
	\end{split}
	\]
\end{itemize}
Thus, the functor \(L \colon \A \to \B\) so defined is the left adjoint of \(R\).
\end{proof}
}

\printbibliography
\end{document}